\documentclass[11pt]{amsart}

\usepackage{times}
\usepackage{tikz}
\usepackage[left=1in,right=1in,top=1.5in,bottom=1.5in]{geometry}
\usepackage{amssymb}
\usepackage{latexsym, amsmath, amscd,amsthm,booktabs,appendix}
\usepackage{graphicx}
\usepackage[percent]{overpic}
\usepackage{units}
\usepackage{hyperref}
\usepackage{color}
\usepackage{xcolor}

\newcommand{\comment}[1]{}

\newtheorem{theorem}{Theorem}

\newtheorem{lemma}[theorem]{Lemma}
\newtheorem{proposition}[theorem]{Proposition}
\newtheorem{corollary}[theorem]{Corollary}

\theoremstyle{definition}
\newtheorem{definition}[theorem]{Definition}
\newtheorem{example}[theorem]{Example}

\newcommand{\Z}{\mathbb{Z}}

\newcommand{\cross}{\times}

\newcommand{\PG}{\mathbb{PG}}
\newcommand{\T}{\mathbb{T}}

\newcommand{\stab}{\operatorname{Stab}}
\newcommand{\orb}{\operatorname{Orbit}}

\newcommand{\PD}{\mathbb{PD}}

\graphicspath{./figs}

\begin{document}
\title{An Enhanced Prime Decomposition Theorem for Knots with Symmetry Information}
     
\author{Matt Mastin}

\begin{abstract} 
We present an enhanced prime decomposition theorem for knots that gives the isotopy classes of composite knots that can be constructed from a given list of prime factors (allowing for the mirroring and orientation reversing for each factor). Underlying the theorem is an algebraic construction that also allows for the computation of the intrinsic symmetries (invertibility, chirality, etc.) of a composite knot from those of the prime factors. We then use this construction to give a table of composite knots through $12$ crossings that can be constructed from prime factors through $9$ crossings. This is more difficult than it might sound because we must take knot symmetries into account when generating the table (the square knot and the granny knot are different, though both are connect sums of two trefoils). The completeness of this table depends on the conjectural additivity of crossing number under the connected sum.
\end{abstract}


\maketitle
\section[Introduction]{Introduction}

Knots can have one of five symmetry types: (+) amphichiral, (-) amphichiral, reversible, no symmetry, and full symmetry. The symmetry type of each prime knot is known through high crossing number \cite{Cerf}, \cite{Kod1},\cite{VIGRE}. For composite knots, new symmetries may arise or be destroyed by combining knots. For example, the trefoil is a reversible (but not amphichiral) knot, so it comes in left- and right-handed versions. However, the square knot (the connect sum of a left- and right-handed trefoil) has full symmetry while the granny knot (the connect sum of two trefoils of the same handedness) is again only reversible. Wilbur Whitten gave a theorem in 1969 \cite{Whit} giving necessary and sufficient conditions for a given symmetry to apply to a composite knot in terms of the symmetry groups of the prime factors of the knot, but the statement of Whitten's theorem encodes much of its content in a very complicated system of indices, making it hard to apply. Here we give an enhanced prime decomposition theorem which allows us to calculate the full set of composites arising from a set of prime factors and their symmetry groups as orbits and stabilizers of a natural group action.

Prime knot tabulation has been given much attention\cite{Kirk1}, \cite{First}, \cite{Rolfsen}, but there is a striking lack of composite knot tables in the literature. This is perhaps somewhat surprising given the many applications in which composite knots arise naturally, such as the knotted strands of DNA discussed by Arsuaga, et al. \cite{Ars1}. From a topological perspective connected sum decompositions of links in $3$-manifolds could be considered inconvenient due to the ambiguity in embedding disjoint $2$-spheres into $S^3$. This ambiguity led $3$-manifold topologists to consider torus decompositions and, in particular, the JSJ-decomposition of link complements in $S^3$ (see, for example, the manuscript of Bonahon and Siebenmann \cite{BS1}, the excellent paper of Ryan Budney \cite{Bud1} or the original papers of Jaco and Shalen \cite{JSJ1} and Johannson \cite{JSJ2}). While this toroidal perspective is convenient for studying many properties of knots and links, the prime decomposition is obfuscated by this viewpoint. For knots the prime summands can be detected from the JSJ-decomposition of the complement and it is this fact that we exploit presently to construct an algebraic structure within which to state an enhanced version of the prime decomposition theorem for knots. We then apply this theorem to the problem of tabulating composite knots. 

It is an open question (and a current research project of the author) to give a relationship between the JSJ-decomposition of a link complement and the link's prime factorization. The main obstruction to this generalization is the increased complexity of the JSJ-decomposition of link complements in $S^3$. 

Section \ref{sect:JSJGraph} will be dedicated to giving a brief survey of the relationship between the JSJ-decomposition of a knot complement and the prime factorization of the knot. Then, in section \ref{sect:PDT} we define the combinatorial object that we will use to describe composite knots. We will be working with knot \emph{diagrams} described by a notation called PD-codes. In section \ref{sect:sym}, discuss the \emph{intrinsic symmetries} of knots in terms of PD-codes. We use the term intrinisic symmetry to refer to the invertibility and/or chirality of a knot. 

In section \ref{sect:decomp} we will state and prove an enhanced prime decomposition theorem for knots, the main theorem of the paper. The classical prime decomposition theorem \footnote{The original paper of Schubert \cite{Shub2} has not been translated, but there is a nice survey of the result by Sullivan \cite{Sul1}} states that a given composite knot has a unique set of prime factors. This enhanced theorem addresses the question of what composite knots are obtained from a given list of prime knots if you are allowed to reverse the orientation and/or mirror each prime factor. In particular, it gives the isotopy classes of knots with a given list of prime factors (up to mirroring the crossings and reversing orientation) as the collection of orbits of a certain group action. Section \ref{sect:CompKnotSyms} is dedicated to the intrinsic symmetry group of a composite knot in terms of the symmetry groups of the prime factors.  We close with a discussion of future work, the most important of which is the generalization of the techniques described here to the case of composite links. The appendix gives a table of composite knots that can be constructed by connected sums of prime knots with $9$ or fewer crossings as well as some details of the computations involved. This table is complete under the conjecture that crossing number is additive under connected sum.


\section[Prime Decomposition Graphs]{Prime Decomposition Graphs}\label{sect:JSJGraph}

The results in this section either appear in or are corollaries of results in a paper of Ryan Budney~\cite{Bud1}.

In order to translate the computation of composite knot symmetries from topology to combinatorics we will define a tree associated to a composite knot. The well-definedness of this construction will depend on a specialized construction related to the JSJ-decomposition of the knot complement. The interested reader is directed to Budney~\cite{Bud1} for full details.

Given a composite knot $K$ we can decompose its complement $C_K$ by embedding solid tori $\cup T_i$ as shown in Figure~\ref{fig:companions}. These tori are the swallow-follow tori corresponding to each prime factor of $K$. It is important to note that the swallow-follow tori are not necessarily the only tori appearing in the full JSJ-decomposition. For example, if one of the prime factors was the Whitehead double of some knot, then we would see additional tori in the JSJ-decomposition which are not swallow-follow tori, but are representing doubling as a satellite operation (that is not connected sum). However, the swallow-follow tori for each prime factor are precisely the tori in the JSJ-decomposition which are adjacent to the knot complement. In fact, the companion link to the $3$-manifold co-bounded by a neighborhood of $K$ and these swallow-follow tori is the link given in the following definition. The companions corresponding to the $3$-manifolds bounded by each swallow-follow tori are simply the prime factors of $K$ as these $3$-manifolds are already knot complements.

%
%
%
%
%
%
%
%
%
%
%


\begin{definition}
We denote by $H^p$ the $(p+1)$-component \textbf{keychain link} (shown in Figure \ref{fig:keychain}) and denote the $i$th component of $H^p$ by $H_i^p$.  The compononents are oriented so that $\operatorname{Lk}(*,H^p_i)=+1$ for $i=1,\ldots,p$, where $\operatorname{Lk}$ is the standard linking number. 
\end{definition}

\begin{figure}
\begin{center}
\includegraphics[height=3cm]{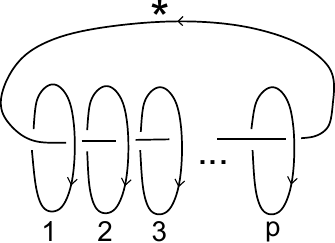}
\caption{\label{fig:keychain}Here we see a $(p+1)$ component keychain link link. Notice that the orientations have been chosen so that all crossings are positive.}
\end{center}
\end{figure}


\begin{figure}
\begin{center}
\includegraphics[height=4cm]{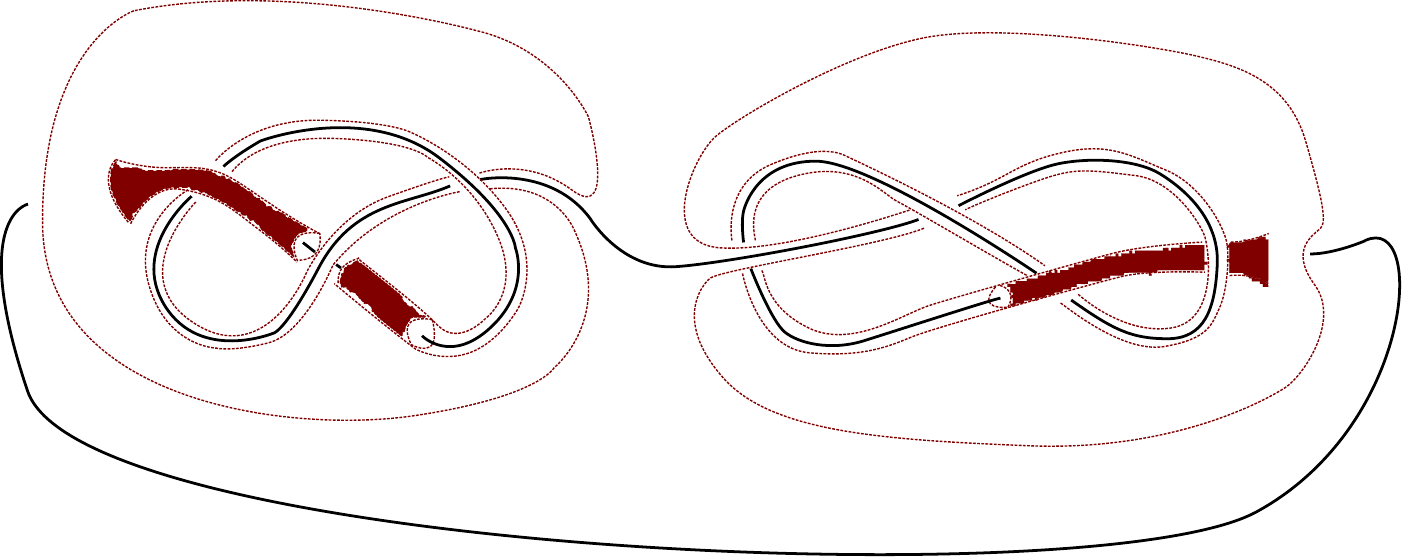}
\caption{\label{fig:companions}This figure, similar to the one appearing in Budney~\cite{Bud1}, shows the JSJ-decomposition of a composite knot. The tori shown are the swallow-follow tori and, in this case, these are the only tori in the JSJ-decomposition. }
\end{center}
\end{figure}


If we restrict our attention to only the swallow-follow tori we can define a simplified version of the JSJ and splicing graphs described by Budney. Such graph constructions seem to be first given by Siebenmann \cite{Sieb1} as well as Eisenbud and Neumann \cite{EN1}. 

\begin{definition}\label{def:PDT}
A \textbf{prime decomposition tree} is a depth one rooted tree $\PG(K)$ with root labeled by the link $H^p$ whose children are ordered vertices labeled by oriented prime knots (see Figure~\ref{fig:comp_tree}). We denote the knot corresponding to vertex $v$ by $\PG(v)$. We say that two prime decomposition trees $\PG$ and $\PG'$ are equivalent, denoted $\PG \sim \PG'$ , if there exists an isomorphism of rooted trees $g:\PG \rightarrow \PG'$ such that the knots $\PG(v)$ and $\PG'(g(v))$ are isotopic. 
\end{definition}

The following result is a restatement of Proposition $4.6$ in Budney \cite{Bud1}.

\begin{proposition}\label{prop:Bud1}
Two knots $K_1$ and $K_2$ are isotopic if and only if $\PG(K_1) \sim \PG(K_2)$.
\end{proposition}

%
%

\begin{figure}
\begin{center}
\includegraphics[height=3cm]{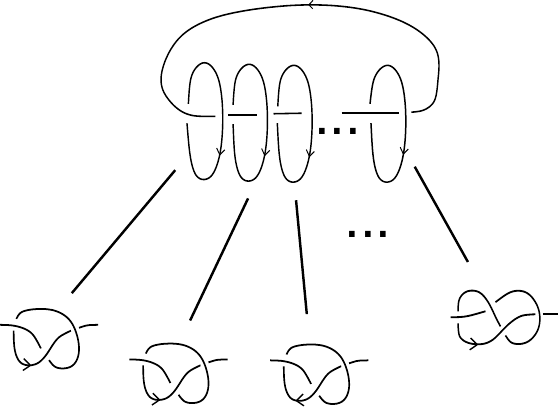}
\caption{\label{fig:comp_tree}An example of a prime decompostion tree.}
\end{center}
\end{figure}

A knot can be recovered from a prime decomposition tree via the splicing operation, but the details of the construction are not needed to proceed. These details can be found in Budney \cite{Bud1}, and we denote the knot obtained via splicing according to $\PG$ by $S(\PG)$.

\begin{theorem}{\label{thm:PGEquiv}}
Let $\PG$ and $\PG'$ be prime decomposition trees, then the knots $S(\PG)$ and $S(\PG')$ are isotopic if and only if $\PG$ and $\PG'$ are equivalent as prime decomposition trees.
\end{theorem}

\begin{proof}
Suppose that $S(\PG) \sim S(\PG')$, then by Proposition~\ref{prop:Bud1} we have that  $\PG(S(\PG)) \sim \PG'(S(\PG))$. So, in particular the knots have the same prime factors. Thus, by the uniqueness of the prime factorization of knots there is a permutation of the leaves of $\PG(S(\PG))$ so that we obtain $\PG'$. 

On the other hand if $\PG \sim \PG'$, then we know immediately that $S(\PG)$ and $S(\PG')$ have the same prime factors and are thus isotopic by the uniqueness of the prime factorization of knots. 
\end{proof}

It is interesting to note that although we use the classical prime decomposition theorem in this proof (for simplicity) one could avoid it and use only the uniqueness of the JSJ-decomposition to obtain a proof of the classical prime decomposition theorem.

\section[Prime Diagram Trees]{Prime Diagram Trees}\label{sect:PDT}

In order to have a unique description of each prime knot type we will from here on be working with particular diagrams instead of equivalence classes of space curves. In particular, we will use Planar Diagram codes (PD-codes) to describe our knot diagrams. PD-codes seem to have been first defined by Dror Bar-Natan for use in the KnotTheory Mathematica Package. The details of the construction have been worked out by the author and can be found in \cite{Ma2}, but the basic definition is included here for convenience.

\begin{definition}\label{def:PD}
Given a knot diagram on an oriented surface $S$, we generate the set of quadruples of the \textbf{PD-code} representing this diagram by the following procedure. For each crossing we include the quadruple of arc labels involved beginning with the incoming under-edge and proceeding around the crossings in the positively oriented direction of $S$ (see Figure~\ref{fig:PD}). We give a positive sign to incoming edges and a negative sign to outgoing edges.

\begin{figure}
\begin{center}
\begin{overpic}[height=4cm]{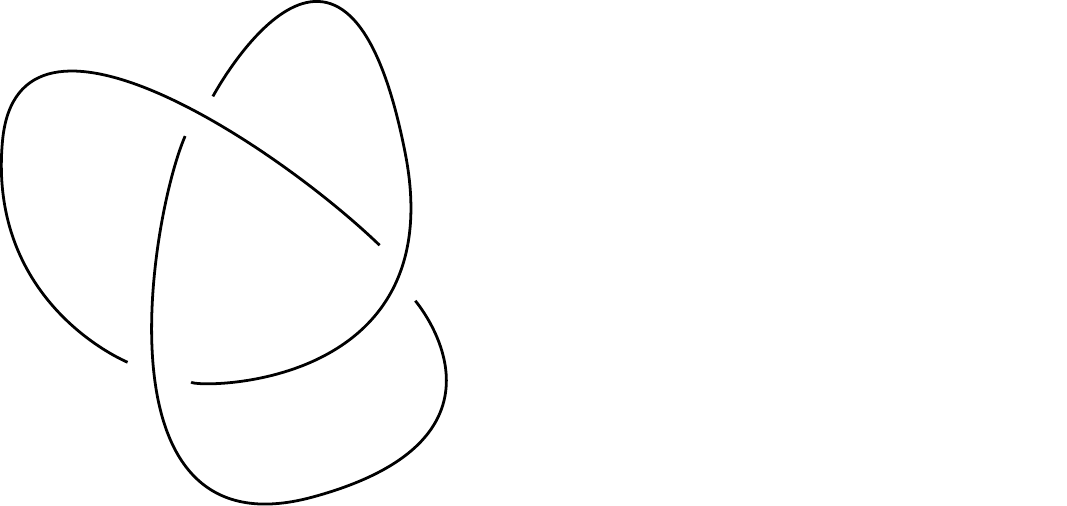}
\put(40,35){$1$}
\put(28,15){$2$}
\put(-3,30){$3$}
\put(28,33){$4$}
\put(35,-1){$5$}
\put(16,21){$6$}
\put(55,27){$\{ [+4,-2,-5,+1],$}
\put(57.5,21){$[+2,-6,-3,+5],$}
\put(57.5,15){$[+6,-4,-1,+3]\}$}
\end{overpic}
\end{center}
\caption{\label{fig:PD} A diagram for $3_1$ and its PD-code. The labels are only single integers here as there is only one component. Note that we may omit directional arrows as the orientation can be inferred from the ordering of the edge labels.}
\end{figure}
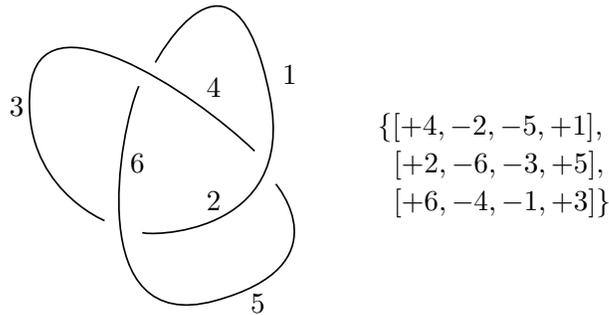
\end{definition}

We now define the combinatorial object that is analogous to the prime decomposition tree except that we label our vertices with PD-codes as opposed to knots.

\begin{definition}\label{def:PDiagT}
A \textbf{prime diagram tree} is a depth $1$ rooted tree whose vertices are ordered and labeled by PD-codes from the prime knot table. We denote the PD-code at vertex $v$ by $\PD(v)$ and the corresponding knot $k(v)$. The vertices respect a chosen ordering on base types so that if $i \leq j$, then $k(v_i) \leq k(v_j)$. We say that two prime diagram trees $T_1$ and $T_2$ are equivalent (denoted $T_1 \sim T_2$) if there is an isomorphism of trees $\phi: T_1 \rightarrow T_2$ such that $k(\phi(v))\sim k(v)$ for every vertex $v \in T_1$.
\end{definition}


Defining a connected sum of PD-codes essentially amounts to choosing a convention for dealing with the indices. We choose the conventions laid out in the following definition.

\begin{definition}\label{def:PDSum}

Consider an ordered list of knot PD-codes $(D_1, D_2)$ where the first has $n_1$ arcs and the second has $n_2$ arcs. We define the \textbf{connected sum of the PD-codes} to be the PD-code obtained by the following procedure (cf. Figure~\ref{fig:sum}).

\begin{enumerate}
	\item add $n_2$ to each \emph{positive} label of $D_1$ and subtract $n_2$ to each negative label of $D_1$ except label $-1$
	\item change the $-1$ to a $-(n_2+1)$ in the quadruple of $D_2$ which also contains $+n_2$
	\item concatenate $D_1$ and $D_2$
\end{enumerate}
\end{definition}

The conventions chosen in Definition \ref{def:PDSum} give the PD code of the connected sum that maintains the edge labels of $D_2$, changes the edge labels of $D_1$ to be consecutive with those in $D_2$, and performs the connected sum along the edges labeled $1$ (as shown in Figure \ref{fig:sum}).


\begin{figure}
\begin{center}
\begin{overpic}[height=6cm]{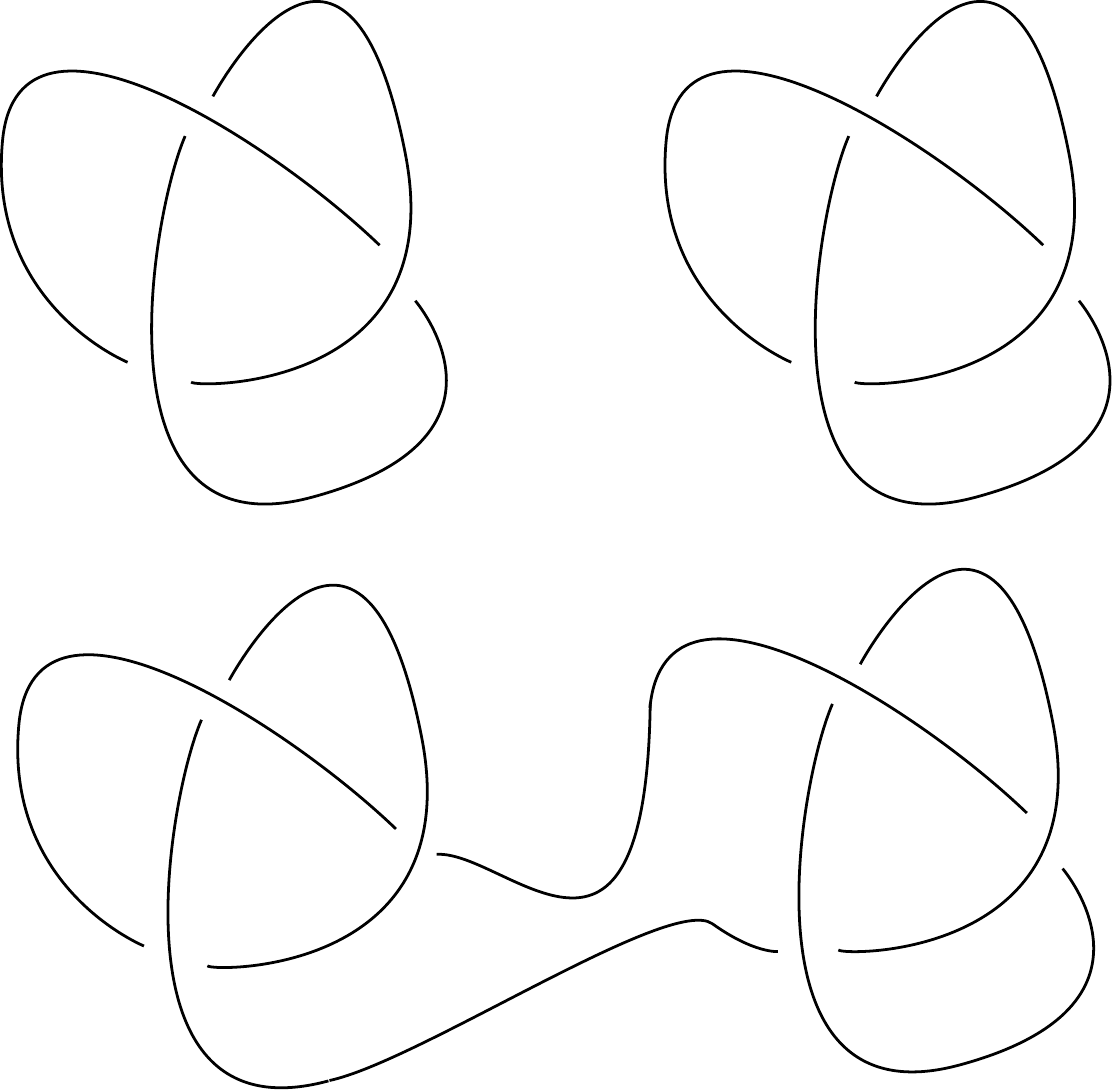}
\put(-18,91){$D_1$}
\put(109,91){$D_2$}
\put(-25,38){$D_1 \# D_2$}

\put(42,60){$1$}
\put(16,72){$2$}
\put(37,90){$3$}
\put(30,60){$4$}
\put(-5,75){$5$}
\put(25,85){$6$}

\put(56,75){$1$}
\put(86,83){$2$}
\put(101,60){$3$}
\put(75,71){$4$}
\put(96,90){$5$}
\put(90,60){$6$}

\put(53,28){$1$}
\put(86,32){$2$}
\put(100,9){$3$}
\put(75,21){$4$}
\put(96,40){$5$}
\put(88,9){$6$}
\put(52,6){$7$}
\put(17,20){$8$}
\put(38,37){$9$}
\put(30,7){$10$}
\put(-6,27){$11$}
\put(27,31){$12$}

\end{overpic}
\end{center}
\caption{\label{fig:sum} The connected sum of the diagrams of two right-handed trefoils.}
\end{figure}
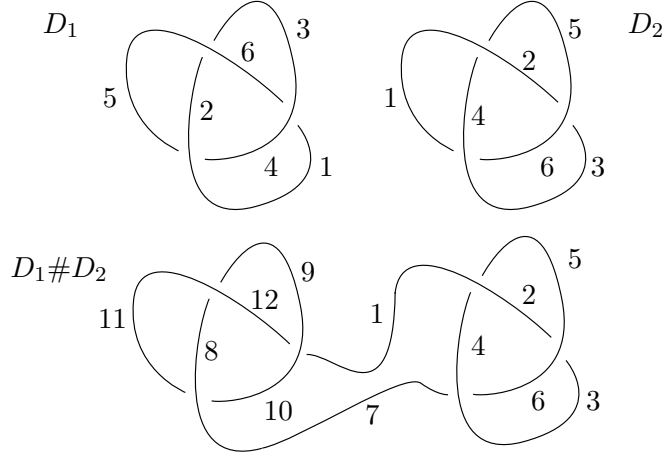

\begin{example}
	Figure \ref{fig:sum} shows the connected sum of two right-handed trefoils. The PD-codes for these two knots are as follows.
 $$\{[+2,-6,-3,+5],[+6,-4,-1,+3],[+4,-2,-5,+1]\} $$  $$\{[+4,-2,-5,+1],[+2,-6,-3,+5],[+6,-4,-1,+3]\}$$ 

	Note that these are the \emph{same} PD-codes as the diagrams in Figure \ref{fig:sum} are related by a rotation of the page (or more appropriately a diffeomorphism of $S^2$ which is isotopic to the identity). We now apply the procedure outlined in Definition \ref{def:PDSum} in order to take the connected sum of these knots letting the first be $D_1$ and the second be $D_2$. Note that $n_1 = n_2 = 6$.
	\begin{enumerate}
		\item We first add $n_2 = 6$ to each label of $D_1$ except the label $-1$. The result is as follows.
			$$\{[+8,-12,-9,+11],[+12,-10,-7,+9],[+10,-8,-11,+7]\}$$
		\item Next, we change the $-1$ to $-(n_2+1) = -7$ in $D_2$ to obtain the following.
			$$\{[+4,-2,-5,+1],[+2,-6,-3,+5],[+6,-4,-7,+3]\}$$ 
		\item Concatenating the PD-codes from $(1)$ and $(2)$ we are arrive at the PD-code below.
			$$\{[+8,-12,-9,+11],[+12,-10,-7,+9],[+10,-8,-11,+7],$$
			$$[+4,-2,-5,+1],[+2,-6,-3,+5],[+6,-4,-7,+3]\}$$
	\end{enumerate}
	It is easily checked that this is the PD-code for the composite knot in Figure \ref{fig:sum}
\end{example}

Since a knot type can be uniquely recovered from a PD-code the equivalence relation defined in Definition \ref{def:PDiagT} is equivalent to knot isotopy and we can therefore consider prime diagram trees instead of prime decomposition trees. We can thus restate Theorem \ref{thm:PGEquiv} as follows.

\begin{proposition}\label{prop:PDiagEquiv}
If $T_1$ and $T_2$ are prime diagram trees, then $k(T_1) \sim k(T_2)$ if and only if $T_1 \sim T_2$.
\end{proposition}

%

\begin{figure}
\begin{center}
\begin{overpic}[height=5cm]{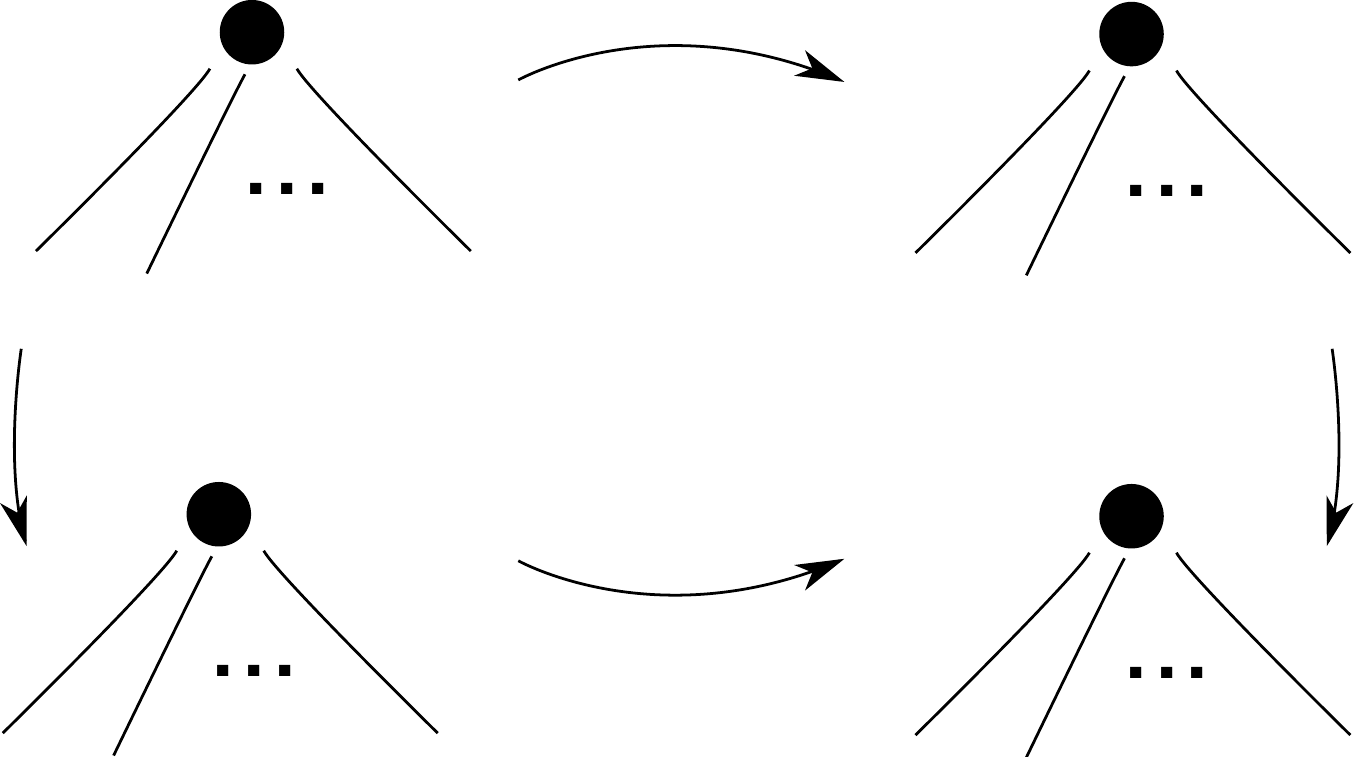}
\put(-3,22){$\phi$}
\put(100,22){$\bar{\phi}$}
\put(49,54){$k$}
\put(49,7){$k$}

\put(-1,33){$D_1$}
\put(8,31){$D_2$}
\put(33,33){$D_n$}

\put(-3,-3){$D_1'$}
\put(6,-5){$D_2'$}
\put(31,-3){$D_n'$}

\put(56,-3){$k(D_1)'$}
\put(71,-5){$k(D_2)'$}
\put(95,-3){$k(D_n)'$}

\put(59,33){$k(D_1)$}
\put(73,31){$k(D_2)$}
\put(95,33){$k(D_n)$}

\end{overpic}
\end{center}
\caption{\label{fig:TreeComDiag} Trees of diagrams mapping to trees of knots.}
\end{figure}
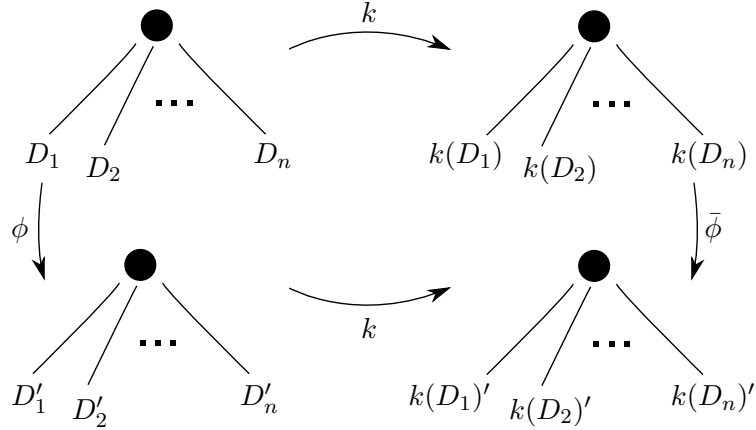

\section{Symmetries of Prime Knots}\label{sect:sym}

The \emph{intrinsic symmetry group} of a link was first defined by Whitten \cite{Whit} and discussed in detail by the author in Cantarella, et al. \cite{VIGRE}. These symmetries are the generalization of invertiblility and chirality and are described by the group given in the following definition.  We give the specialized definition corresponding to a knot.


\begin{definition}
	The group of possible intrinsic symmetries of a knot is given by the group $\Gamma = \Z_2 \times \Z_2$. We will  describe $\Z_2$ as the multiplicative group $\{1,-1\}$ and we will write an element $\gamma \in \Gamma$ as $(\epsilon_0, \epsilon_1)$. If $\epsilon_0 = -1$, then the element corresponds to mirroring the knot. Similarly, if $\epsilon_1=-1$ then the element corresponds to changing the knots orientation. 
\end{definition}

The action of the intrinsic symmetry group on PD-codes (for any link) is described in detail by the author in \cite{Ma2}. Lemma 27 of \cite{Ma2} ensures the existence of a preferred list of PD-codes, one for each knot type, which we will refer to from now on by \textbf{the prime knot table}. The PD-codes from the preferred list have the property that only the trivial element of $\Gamma$ acts trivially on the PD-code. We will use the prime knot table to define a combinatorial object analogous to the prime decomposition tree that will play a key role in the computation of symmetries and the tabulation of composite knots.

We now introduce our conventions for indexing the prime factors of a knot and define the action of $\Gamma$ on a prime diagram tree.

\begin{definition}\label{def:T}
Let $\mathbb{T}$ be the collection of prime diagram trees whose vertices are labeled by diagrams from the prime knot table.
\end{definition}

\begin{definition}
Given a prime decomposition tree $\PG(K)$ and $\gamma \in \Gamma$ we define $\gamma(\PG(K))$ to be the tree whose ordered leaves are the knots $\gamma(K_i)$ where the $K_i$ are the labels of $\PG(K)$.
\end{definition}


\begin{definition}\label{def:FactorList}
A \textbf{base prime factor list} $P=\{(D_i,n_i)\}_{i=1}^l$ is a set of PD-codes from the prime knot table along with multiplicities. We say that $P$ is the base prime factor list for a knot $k$ if the base types of the prime factors of $k$ appear in $P$ with the correct multiplicities.
\end{definition}

\begin{definition}\label{def:TP}
The collection of all prime diagram trees whose leaves are labeled by knots whose base types are exactly the base prime factor list $P$ will be denoted $\T(P)$.
\end{definition}

\begin{definition}\label{def:X}
Let $X(P)$ be the set $\{\cross_{i=1}^l \Gamma^{n_i}\}$.
\end{definition}

The set $X(P)$ describes the possible choices for taking the connected sum of $n$ prime knots.

\begin{lemma}\label{lemma:XisTP}
$X(P)$ is in bijection with $\T(P)$.
\end{lemma}

\begin{proof}
The correspondence is given by associating $x=((x_{1,1},\ldots,x_{1,n_1}),\ldots,(x_{l,1},\ldots,x_{l,n_l}))$ with the prime diagram tree whose children are $x_{1,1}D_1,\ldots,x_{1,n_1}D_1,\ldots,x_{l,1}D_l,\ldots,x_{l,n_l}D_l$ (cf. Figure~\ref{fig:Tree}). We will denote this tree by $\T(x)$.

\begin{figure}
\begin{center}
\begin{overpic}{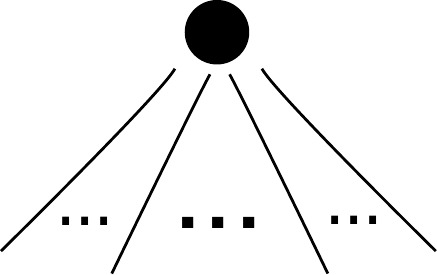}
\put(-14,-5){$x_{1,1}D_1$}
\put(16,-8){$x_{1,n_1}D_1$}
\put(66,-8){$x_{l,1}D_l$}
\put(95,-5){$x_{l,n_l}D_l$}
\end{overpic}
\end{center}
\caption{\label{fig:Tree} The prime diagram tree associated to $P=\{(D_i,n_i)\}_{i=1}^l$ and $x=((x_{1,1},\ldots,x_{1,n_1}),\ldots,(x_{l,1},\ldots,x_{l,n_l}))$.}
\end{figure}
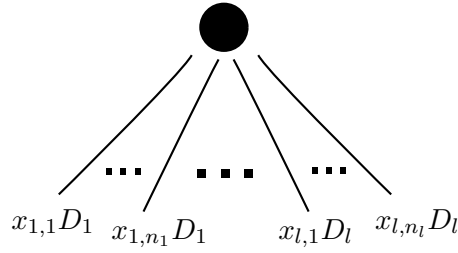

\end{proof}

\begin{definition}\label{def:GammaP}
We define the group $\Gamma(P)$ by
$$\Gamma(P):= \oplus_{i=1}^l \left[ \left( \oplus_{n_i} \Gamma \right) \rtimes S_{n_i} \right]$$
 
and the group $\Sigma(P)$ by

$$\Sigma(P):= \oplus_{i=1}^l \left[ \left( \oplus_{n_i} \Sigma(k_i) \right) \rtimes S_{n_i} \right]$$
\end{definition}

%
%
%
%
%
%
%
%

We can let the group $\Gamma(P)$ act on $X(P)$ to change between different choices of orientations and handedness of each prime factor in the connected sum. This action is defined as follows.

\begin{definition}\label{def:GammaOnX}
There is a natural action of $\Gamma(P)$ on $X(P)$ given by the following.

$$((\gamma_{1,1},\ldots,\gamma_{1,n_1},p_1),\ldots,(\gamma_{l,1},\ldots,\gamma_{l,n_l},p_l))*((x_{1,1},\ldots,x_{1,n_1}),\ldots,(x_{l,1},\ldots,x_{l,n_l}))$$
$$=((\gamma_{1,1}x_{1,p_1(1)},\ldots,\gamma_{1,n_1}x_{1,p_1(n_1)}),\ldots,(\gamma_{l,1}x_{1,p_l(1)},\ldots,\gamma_{l,n_l}x_{l,p_l(n_l)}))$$
\end{definition}


Therefore,  we have an induced action of $\Gamma(P)$ on $\T(P)$.  

\section[An Enhanced Prime Decomposition Theorem]{An Enhanced Prime Decomposition Theorem}\label{sect:decomp}

We will now lead up to the main theorem of the paper. Theorem \ref{theorem:SigPOnX} is an enhanced version of the classical prime decomposition theorem for knots in the sense that it gives an explicit way of constructing all composite knots whose \emph{base} prime factor list is given. 

\begin{proposition}\label{prop:SigmaOrbs}
Let $x$ and $x'$ be elements of $X(P)$ for some prime factor list $P$. Then $\T(x)$ and $\T(x')$ are equivalent as prime diagram trees if and only if there exists $\sigma \in \Sigma(P)$ such that $\sigma(x)=x'$ under the action of $\Sigma(P)$ on $X(P)$.
\end{proposition}

\begin{proof}
First suppose that $\T(x)$ and $\T(x')$ are equivalent as prime diagram trees. Then, there is an isomorphism of trees $\phi: \T(x) \rightarrow \T(x')$ with corresponding isotopies between $k(v)$ and $k(\phi(v))$. Let $x=((x_{1,1},\ldots,x_{1,n_1}),\ldots,(x_{l,1},\ldots,x_{l,n_l}))$ and $x'=((x_{1,1}',\ldots,x_{1,n_1}'),\ldots,(x_{l,1}',\ldots,x_{l,n_l}'))$ and not that $\phi$ can only permute elements within each tuple since only knots of the same base type may be interchanged. Moreover, if $\phi(v) = v'$, then $k(v)$ and $k(v')$ must be isotopic knots and are therefore related by an element in $\Sigma(k(v)) = \Sigma(k(v'))$ (recall that knots of the same base type have the same intrinsic symmetry group). Let $p_i$ be the permutations induced by $\phi$ on each collection of leaves with a common base type and consider the element $\sigma \in \Sigma(P)$ that is defined as follows.

$$\sigma = ((x_{1,1}' * x_{1,p_1(1)}^{-1},\ldots,x_{1,n_1}' * x_{1,p_1(n_1)}^{-1},p_1),\ldots,(x_{l,1}' * x_{l,p_l(1)}^{-1},\ldots,x_{l,n_1}' * x_{l,p_l(n_1)}^{-1},p_l))$$

We claim that $\sigma(x) = x'$ and this can be verified by the following computation. \\

\begin{tabular}{lcl}
$\sigma(x)$ &  $=$ & $((x_{1,1}' * x_{1,p_1(1)}^{-1},\ldots,x_{1,n_1}' * x_{1,p_1(n_1)}^{-1},p_1),\ldots,(x_{l,1}' * x_{l,p_l(1)}^{-1},\ldots,x_{l,n_1}' * x_{l,p_l(n_1)}^{-1},p_l)) $ \\ \\
 & & $ *~~ ((x_{1,1},\ldots,x_{1,n_1}),\ldots,(x_{l,1},\ldots,x_{l,n_l})) $\\ \\

 & $ = $ & $ ((x_{1,1}' * x_{1,p_1(1)}^{-1} * x_{1,p_1(1)},\ldots,x_{1,n_1}' * x_{1,p_1(n_1)}^{-1} * x_{1,p_1(n_1)}),\ldots,$\\ \\
 & & $(x_{l,1}' * x_{l,p_l(1)}^{-1} * x_{l,p_l(1)},\ldots,x_{l,n_1}' * x_{l,p_l(n_1)}^{-1} * x_{l,p_l(n_l)}))$ \\ \\

 & $ = $ & $((x_{1,1}',\ldots,x_{1,n_1}'),\ldots,(x_{l,1}',\ldots,x_{l,n_l}')) $ \\ \\
 & $=$ & $x'$. \\
\end{tabular}

It remains to show that $\sigma$ is in fact an element of $\Sigma(P)$. Since $p_i$ can only permute leaves that correspond to isotopic knots we know that $x_{j,k}$ and $x_{j,p_i(k)}$ must be elements in the intrinsic symmetry of the base type corresponding to those leaves. It follows that $x_{j,k} * x_{j,p_i(k)}^{-1}$ is also an element in the intrinsic symmetry group and thus $\sigma \in \Sigma(P)$.

Now conversely suppose that there is some $\sigma \in \Sigma(P)$ such that $\sigma(x)=x'$. If $$\sigma = ((\gamma_{1,1},\ldots,\gamma_{1,n_1},p_1),\ldots,(\gamma_{l,1},\ldots,\gamma_{l,n_l},p_l))$$ then permuting the leaves of $\T(x)$ by the permutation $p_i$ only permutes leaves of the same base type and since each $\gamma_{i,j} \in \Sigma(k_i)$ we have that permuted vertices are equivalent PD-codes (or we could say isotopic knots). Thus, $\sigma$ induces an isomorphism of decomposition trees and so $\T(x)\sim \T(x')$.

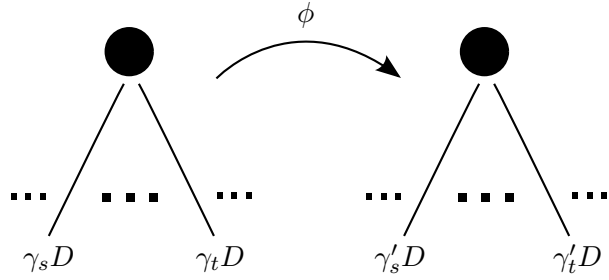
\begin{figure}
\begin{center}
\begin{overpic}{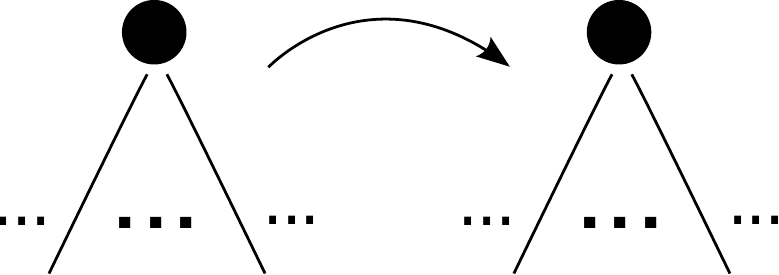}
\put(48,36){$\phi$}
\put(2,-5){$\gamma_{s}D$}
\put(31,-5){$\gamma_{t}D$}
\put(61,-5){$\gamma_{s}'D$}
\put(91,-5){$\gamma_{t}'D$}
\end{overpic}
\end{center}
\caption{\label{fig:TreePermute} A map of prime diagram trees.}
\end{figure}
\end{proof}


\begin{lemma}\label{lemma:kSurjects}
The map $k$ of Definition \ref{def:PDiagT} descends to a surjective map from the $\Sigma(P)$ orbits of $\mathbb{T}(P)$ to knot types whose base prime factor list is $P$.
\end{lemma}

\begin{proof}
We first show that $k$ descends to a map on the $\Sigma(P)$ orbits of $\mathbb{T}(P)$. Suppose $T$ and $T'$ are in the same orbit, then there exists a $\sigma \in \Sigma(P)$ such that $\sigma T = T'$. Thus, by Proposition~\ref{prop:SigmaOrbs} we have that $T \sim T'$. So by Proposition~\ref{prop:PDiagEquiv}, $k(T) \sim k(T')$ and we see that $k$ descends to a map on $\Sigma(P)$ orbits.

Since every knot has a prime factorization and the action of $\Gamma$ on each base type is transitive, we see that $k$ is surjective.
\end{proof}

\begin{theorem}\label{theorem:SigPOnX}(Enchanced Prime Decomposition Theorem)
The orbits of the $\Sigma(P)$ action on $X(P)$ are in bijection with the isotopy classes of knots whose base prime factor list is $P$.
\end{theorem}

\begin{proof}
First note that the $\Sigma(P)$ orbits of $X(P)$ are the same as the $\Sigma(P)$ orbits of $\mathbb{T}(P)$ by construction. Proposition~\ref{prop:SigmaOrbs} shows that the collection of $\Sigma(P)$ orbits on $\mathbb{T}(P)$ is the partition associated to the equivalence on prime decomposition trees. Thus, the result follows from Proposition~\ref{prop:PDiagEquiv}.

\end{proof}

\begin{definition}
We can now define $\orb(K)$ to be the orbit in $X(P)$ which corresponds to the knot $K$ under the map $k$.
\end{definition}

\begin{example}\label{example:31_31Orbs}
Consider the base prime factor list $P=\{(3_1,2)\}$ where $3_1$ denotes the PD-code of the standard diagram of the trefoil of Figure~\ref{fig:PD}. This corresponds to taking the connected sum of $2$ trefoils. The trefoil is an invertible, chiral knot so $\Sigma(3_1)=\{(1,1),(1,-1)\}$. Therefore, $$\Sigma(\{(3_1,2)\})=(\Sigma(3_1) \oplus \Sigma(3_1)) \rtimes S_2 = (\{(1,1),(1,-1)\} \oplus \{(1,1),(1,-1)\}) \rtimes S_2,$$ and $$X(\{(3_1,2)\})=\Gamma \cross \Gamma.$$ The orbits of the action of $\Sigma(\{(3_1,2)\})$ on $X(\{(3_1,2)\})$ are shown in Table~\ref{figure:TrefoilClasses}. There are $3$ isotopy classes of composite knots that can be constructed from $2$ trefoils. They are the right and left handed granny knots and the square knot.

\begin{figure}
\begin{center}
\includegraphics[height=3cm]{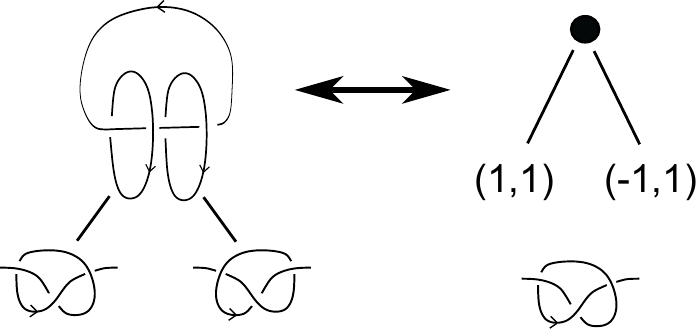}
\caption{\label{figure:SquareTree} A prime decomposition tree and its associated tree labeled by $\Gamma$.}
\end{center}
\end{figure}

\begin{table}
\begin{center}
\begin{tabular}{cc}
\toprule
Composite & Orbit\\
\toprule
\midrule
 Granny Knot & $\begin{array}{cc} ((1,1),(1,1))&((1,-1),(1,1))\\((1,1),(1,-1))&((1,-1),(1,-1)) \end{array}$\\
\midrule
Square Knot & $\begin{array}{cc} ((1,1),(-1,1))&((1,-1),(-1,1))\\((1,1),(-1,-1))&((1,-1),(-1,-1)) \\ ((-1,1),(1,1))&((-1,-1),(1,1))\\ ((-1,1),(1,-1)) & ((-1,-1),(1,-1)) \end{array}$\\
\midrule
Granny Knot & $\begin{array}{cc} ((-1,1),(-1,1))&((-1,-1),(-1,1))\\((-1,1),(-1,-1))&((-1,-1),(-1,-1)) \end{array}$\\	
\bottomrule
\end{tabular}
\end{center}
\caption{\label{figure:TrefoilClasses} The orbits of the action of $\Gamma$ corresponding to the isotopy classes of $3_1 \# 3_1$ of Example~\ref{example:31_31Orbs}.}
\end{table}

\end{example}

\section[Symmetries of Composite Knots]{Symmetries of Composite Knots}\label{sect:CompKnotSyms}

We now turn to the task of computing the intrinsic symmetries of a composite knot from those of its prime factors. Table~\ref{figure:SymNums} gives the occurrences of each symmetry type among the $544$ composite knots through $12$ crossings. The results in this section justify the computation which is explained in section~\ref{sect:CompTab}.

\begin{definition}\label{def:Delta}
We define $\Delta$ to be the following subgroup of $\Gamma(P)$.
$$\Delta:=\{((\gamma,\ldots,\gamma,p_1),\ldots,(\gamma,\ldots,\gamma,p_l)) | \gamma \in \Gamma, p_i \in S_{n_i}\}$$

\end{definition}

The following lemma is immediate.

\begin{lemma}\label{lemma:DeltatoGamma}
The map $\pi:\Delta \rightarrow \Gamma$ given by 
$$((\gamma,\ldots,\gamma,p_1),\ldots,(\gamma,\ldots,\gamma,p_l)) \mapsto \gamma$$
is a surjection.
\end{lemma}

By restricting the action of $\Gamma(P)$ to the subgroup $\Delta$ and factoring through the map of  Lemma~\ref{lemma:DeltatoGamma} we have a well-defined action of $\Gamma$ on prime diagram trees which corresponds with the action of $\Gamma$ on knots.


The following theorem gives the symmetry group of a composite knot.

\begin{theorem}\label{theorem:Syms}
$\Sigma(K) = \pi(\Delta \cap \stab{(\orb(K))})$
\end{theorem}

\begin{proof}
First note that if $\sigma \in \Sigma(K)$ then the element $\bar{\sigma}:=((\sigma,\ldots,\sigma,id),\ldots,(\sigma,\ldots,\sigma,id)) \in \Delta$ must stabilize $\orb(K)$ by Theorem~\ref{theorem:SigPOnX} as $K \sim \sigma(K)$. Since $\pi(\bar{\sigma})=\sigma$ we have that $\Sigma(k) \subset \pi(\Delta \cap \stab{(\orb(K))})$.

Now suppose that $\gamma \in \pi(\Delta \cap \stab{(\orb(K))})$.  We must show that $K \sim \gamma(K)$. First note that elements of the form $(((1,1),\ldots,(1,1),q_1),\ldots,((1,1),\ldots,(1,1),q_k))$ always act trivially for any choice of the permutations $q_i$ since the connected sum of knots is commutative. Thus we may assume that $((\gamma,\ldots,\gamma,id),\ldots,(\gamma,\ldots,\gamma,id)) \in \stab{(\orb(K))}$. But, $\pi(((\gamma,\ldots,\gamma,id),\ldots,(\gamma,\ldots,\gamma,id)))=\gamma$. Thus, $\gamma \in \Sigma(K)$.
\end{proof}


This provides an alternate proof of Theorem 2 of Whitten \cite{Whit}, which was obtained without using the JSJ-decomposition. Whitten’s version of the theorem gave conditions for an element of $\Gamma$ to be a symmetry of a composite knot, encoding the action of $\Gamma(P)$ on $X(P)$ in a complicated system of indices. By exposing the underlying algebra, our version is simpler to state and more amenable to computer calculation of symmetry groups.

There are several immediate corollaries to Theorem~\ref{theorem:Syms} which predict symmetries of a composite knot from the symmetries of the prime factors. We first discuss a generalization of the square knot from Example~\ref{example:31_31Orbs}.

\begin{definition}
$K$ is a \textbf{generalized square knot} if there exist $\gamma_1, \gamma_2 \in \Gamma$ so that the prime factors for $K$ are 
$$\{K_1,\gamma_1 K_1,\ldots,K_1,\gamma_1 K_1\}$$ 
where $\Sigma(K_1)=<\gamma_2>$ and $\Gamma=<\gamma_1,\gamma_2>$.
\end{definition}

\begin{corollary}\label{cor:squareknot}
If $K$ is a generalized square knot, then $\Sigma(K)=\Gamma$.
\end{corollary}

Corollary~\ref{cor:squareknot} could be thought of as an application of the following which gives sufficient conditions for a $2$ factor summand to admit a symmetry.

\begin{corollary}
If $K=K_1 \# \gamma K_1$, then $\gamma \in \Sigma(K)$.
\end{corollary}

It is also immediate that a connected sum of knots with full symmetry will also have full symmetry.

\begin{corollary}
If $K$ is a composite knot with prime factors $\{K_1,\ldots,K_n\}$ such $\Sigma(K_i)=\Gamma$ for all $i$, then $\Sigma(K)=\Gamma$.
\end{corollary}

We can produce a knot with full symmetry by taking the connected sum of each flavor of base type. 

\begin{corollary}
If $K=K_1 \# \gamma_1 K_1 \# \gamma_2 K_1 \# \gamma_3 K_1$ with $\gamma_1, \gamma_2, \gamma_3$ all distinct, then $\Sigma(K)=\Gamma$.
\end{corollary}

If $K$ has no symmetry, then iterated connected sums of $K$ will also have no symmetry.

\begin{corollary}
If $K$ has no symmetry. Then $K \# \cdots \# K$ has no symmetry.
\end{corollary}

\begin{table}
\begin{center}
\begin{tabular}{cc}
\toprule
Symmetry Type & Number of Occurrences\\
\toprule
\toprule
No Symmetry & $20$\\

(+) Amphichiral Symmetry & $0$\\

Invertible Symmetry & $506$\\	

(-) Amphichiral Symmetry & $2$\\

Full Symmetry & $16$\\
\bottomrule
\end{tabular}
\end{center}
\caption{\label{figure:SymNums} The number of each symmetry type among the $544$ composite knots of up to $12$ crossings (up to the conjectural additivity of crossing number under connected sum).}
\end{table}

\begin{table}
\begin{center}
\begin{tabular}{cccccc}
\toprule
Crossing Number & No Symmetry & (+) Amp & Invertible & (-) Amph & Full Symmetry\\
\toprule
\toprule
6 &0&0 & 2& 0 & 1\\
7 & 0&0 & 2& 0& 0\\
8 & 0&0 & 8& 0& 1\\
9 & 0&0 & 18&0 &0 \\
10 &0 &0 & 42&0 &4 \\
11 &4 & 0& 120&0 &0 \\
12 &16 & 0& 314&2 &10 \\
\bottomrule
\end{tabular}
\end{center}
\caption{\label{figure:SymsByCross} The number of each symmetry type among the $544$ composite knots of up to $12$ crossings (up to the conjectural additivity of crossing number under connected sum)  by crossing number.}
\end{table}

\section[Future Directions]{Future Directions}

The obvious next step in this theory is to generalize to the case of links. Not only would a composite link table be a useful addition to the literature, but the underlying topology and combinatorics are very interesting. The first goal for developing such a theory is to understand how the prime decomposition of a link is ascertained from the JSJ-decomposition of the link complement. In the case of knots this is straight-forward: you simply look at the JSJ tori which are adjacent to the knot compliment. The prime factors are the companions to the $3$-cells which bound these tori and are opposite the knot complement. For links the situation seems to be much more complicated because there could be link adjacent tori in the JSJ-decomposition which do not correspond to prime summands. In addition, it is not clear what the analogous object to the prime decomposition tree (cf. Definition~\ref{def:PDT}) should be. These objects must contain enough combinatorial information to encode which component of each prime link is involved in the connected sum. It is also unclear what group naturally acts on these objects in a analogous way to the action of $\Sigma(P)$ on $X(P)$. However, once these objects are understood, the analogous enhanced prime decomposition theorem would be immediate using the same techniques used for knots due to the strong uniqueness properties of the JSJ-decomposition.


\bibliographystyle{alpha}
\bibliography{../bibtex/refs}

\appendix
\section{The Composite Knot Table}

In this appendix we present the table of composite links through $12$ crossings along with their symmetry groups. More precisely, this is the table of composite knots through $12$ crossings that can be constructed by taking connected sums of prime knots through $9$ crossings. If crossing number is indeed additive under the connected sum, then these two descriptions of the table coincide and the table is complete through $12$ crossings. The data was computed using a \texttt{Mathematica} package written by the author and Jason Cantarella.  The tables presented here match and serve as a verification for the table of composite knots found in Cantarella, Rawdon, and LaPointe ~\cite{TightComps}. The dictionary between the classical symmetry names and the intrinsic symmetry group as a subgroup of $\Gamma$ is given in Table \ref{figure:SymDic}.

\begin{table}
\begin{center}
\begin{tabular}{cc}
\toprule
Symmetry Type & Subgroup of $\Gamma$\\
\toprule
\toprule
No Symmetry & $\{(1,1)\}$\\

(+) Amphichiral Symmetry & $\{(1,1),(-1,1)\}$\\

Invertible Symmetry & $\{(1,1),(1,-1)\}$\\	

(-) Amphichiral Symmetry & $\{(1,1),(-1,-1)\}$\\

Full Symmetry & $\Gamma$\\
\bottomrule
\end{tabular}
\end{center}
\caption{\label{figure:SymDic}  Each intrinsic symmetry type of a knot and the corresponding subgroup of $\Gamma$. }
\end{table}

\begin{table}[h]
 \caption{\label{CompositeKnotTable 0} Composite Knot Types, Part 1 of 6}

\begin{center}
\hfill\begin{tabular}[t]{l@{\hspace{0.1in}}r} \toprule
Knot & Symmetry\\ \midrule
$3_{1} \# 3_{1}$ & Invertible \\ 

$3_{1} \# 3_{1}^m$ & Full \\ 

$3_{1}^m \# 3_{1}^m$ & Invertible \\ \addlinespace[0.491339876555172em]\cmidrule(r{0.5em}l{0.5em}){1-2}\addlinespace[0.491339876555172em]

$3_{1} \# 4_{1}$ & Invertible \\ 

$3_{1}^m \# 4_{1}$ & Invertible \\ \addlinespace[0.491339876555172em]\cmidrule(r{0.5em}l{0.5em}){1-2}\addlinespace[0.491339876555172em]

$3_{1} \# 5_{1}$ & Invertible \\ 

$3_{1} \# 5_{1}^m$ & Invertible \\ 

$3_{1}^m \# 5_{1}$ & Invertible \\ 

$3_{1}^m \# 5_{1}^m$ & Invertible \\ \addlinespace[0.491339876555172em]

$3_{1} \# 5_{2}$ & Invertible \\ 

$3_{1} \# 5_{2}^m$ & Invertible \\ 

$3_{1}^m \# 5_{2}$ & Invertible \\ 

$3_{1}^m \# 5_{2}^m$ & Invertible \\ \addlinespace[0.491339876555172em]

$4_{1} \# 4_{1}$ & Full \\ \addlinespace[0.491339876555172em]\cmidrule(r{0.5em}l{0.5em}){1-2}\addlinespace[0.491339876555172em]

$3_{1} \# 6_{1}$ & Invertible \\ 

$3_{1} \# 6_{1}^m$ & Invertible \\ 

$3_{1}^m \# 6_{1}$ & Invertible \\ 

$3_{1}^m \# 6_{1}^m$ & Invertible \\ \addlinespace[0.491339876555172em]

$3_{1} \# 6_{2}$ & Invertible \\ 

$3_{1} \# 6_{2}^m$ & Invertible \\ 

$3_{1}^m \# 6_{2}$ & Invertible \\ 

$3_{1}^m \# 6_{2}^m$ & Invertible \\ \addlinespace[0.491339876555172em]

$3_{1} \# 6_{3}$ & Invertible \\ 

$3_{1}^m \# 6_{3}$ & Invertible \\ \addlinespace[0.491339876555172em]

$4_{1} \# 5_{1}$ & Invertible \\ 

$4_{1} \# 5_{1}^m$ & Invertible \\ \addlinespace[0.491339876555172em]

$4_{1} \# 5_{2}$ & Invertible \\ 

$4_{1} \# 5_{2}^m$ & Invertible \\ \addlinespace[0.491339876555172em]

$3_{1} \# 3_{1} \# 3_{1}$ & Invertible \\ 

$3_{1} \# 3_{1} \# 3_{1}^m$ & Invertible \\ 

$3_{1} \# 3_{1}^m \# 3_{1}^m$ & Invertible \\ 

$3_{1}^m \# 3_{1}^m \# 3_{1}^m$ & Invertible \\ 
\bottomrule
\end{tabular}
\hfill\begin{tabular}[t]{l@{\hspace{0.1in}}r} \toprule
Knot & Symmetry\\ \midrule
$3_{1} \# 7_{1}$ & Invertible \\ 

$3_{1} \# 7_{1}^m$ & Invertible \\ 

$3_{1}^m \# 7_{1}$ & Invertible \\ 

$3_{1}^m \# 7_{1}^m$ & Invertible \\ \addlinespace[0.486673912699349em]

$3_{1} \# 7_{2}$ & Invertible \\ 

$3_{1} \# 7_{2}^m$ & Invertible \\ 

$3_{1}^m \# 7_{2}$ & Invertible \\ 

$3_{1}^m \# 7_{2}^m$ & Invertible \\ \addlinespace[0.486673912699349em]

$3_{1} \# 7_{3}$ & Invertible \\ 

$3_{1} \# 7_{3}^m$ & Invertible \\ 

$3_{1}^m \# 7_{3}$ & Invertible \\ 

$3_{1}^m \# 7_{3}^m$ & Invertible \\ \addlinespace[0.486673912699349em]

$3_{1} \# 7_{4}$ & Invertible \\ 

$3_{1} \# 7_{4}^m$ & Invertible \\ 

$3_{1}^m \# 7_{4}$ & Invertible \\ 

$3_{1}^m \# 7_{4}^m$ & Invertible \\ \addlinespace[0.486673912699349em]

$3_{1} \# 7_{5}$ & Invertible \\ 

$3_{1} \# 7_{5}^m$ & Invertible \\ 

$3_{1}^m \# 7_{5}$ & Invertible \\ 

$3_{1}^m \# 7_{5}^m$ & Invertible \\ \addlinespace[0.486673912699349em]

$3_{1} \# 7_{6}$ & Invertible \\ 

$3_{1} \# 7_{6}^m$ & Invertible \\ 

$3_{1}^m \# 7_{6}$ & Invertible \\ 

$3_{1}^m \# 7_{6}^m$ & Invertible \\ \addlinespace[0.486673912699349em]

$3_{1} \# 7_{7}$ & Invertible \\ 

$3_{1} \# 7_{7}^m$ & Invertible \\ 

$3_{1}^m \# 7_{7}$ & Invertible \\ 

$3_{1}^m \# 7_{7}^m$ & Invertible \\ \addlinespace[0.486673912699349em]

$4_{1} \# 6_{1}$ & Invertible \\ 

$4_{1} \# 6_{1}^m$ & Invertible \\ \addlinespace[0.486673912699349em]

$4_{1} \# 6_{2}$ & Invertible \\ 

$4_{1} \# 6_{2}^m$ & Invertible \\ \addlinespace[0.486673912699349em]

$4_{1} \# 6_{3}$ & Full \\ \addlinespace[0.486673912699349em]

$5_{1} \# 5_{1}$ & Invertible \\ 
\bottomrule
\end{tabular}
\hfill\begin{tabular}[t]{l@{\hspace{0.1in}}r} \toprule
Knot & Symmetry\\ \midrule
$5_{1} \# 5_{1}^m$ & Full \\ 

$5_{1}^m \# 5_{1}^m$ & Invertible \\ \addlinespace[0.455092817969205em]

$5_{1} \# 5_{2}$ & Invertible \\ 

$5_{1} \# 5_{2}^m$ & Invertible \\ 

$5_{1}^m \# 5_{2}$ & Invertible \\ 

$5_{1}^m \# 5_{2}^m$ & Invertible \\ \addlinespace[0.455092817969205em]

$5_{2} \# 5_{2}$ & Invertible \\ 

$5_{2} \# 5_{2}^m$ & Full \\ 

$5_{2}^m \# 5_{2}^m$ & Invertible \\ \addlinespace[0.455092817969205em]

$3_{1} \# 3_{1} \# 4_{1}$ & Invertible \\ 

$3_{1} \# 3_{1}^m \# 4_{1}$ & Full \\ 

$3_{1}^m \# 3_{1}^m \# 4_{1}$ & Invertible \\ \addlinespace[0.455092817969205em]\cmidrule(r{0.5em}l{0.5em}){1-2}\addlinespace[0.455092817969205em]

$3_{1} \# 8_{1}$ & Invertible \\ 

$3_{1} \# 8_{1}^m$ & Invertible \\ 

$3_{1}^m \# 8_{1}$ & Invertible \\ 

$3_{1}^m \# 8_{1}^m$ & Invertible \\ \addlinespace[0.455092817969205em]

$3_{1} \# 8_{2}$ & Invertible \\ 

$3_{1} \# 8_{2}^m$ & Invertible \\ 

$3_{1}^m \# 8_{2}$ & Invertible \\ 

$3_{1}^m \# 8_{2}^m$ & Invertible \\ \addlinespace[0.455092817969205em]

$3_{1} \# 8_{3}$ & Invertible \\ 

$3_{1}^m \# 8_{3}$ & Invertible \\ \addlinespace[0.455092817969205em]

$3_{1} \# 8_{4}$ & Invertible \\ 

$3_{1} \# 8_{4}^m$ & Invertible \\ 

$3_{1}^m \# 8_{4}$ & Invertible \\ 

$3_{1}^m \# 8_{4}^m$ & Invertible \\ \addlinespace[0.455092817969205em]

$3_{1} \# 8_{5}$ & Invertible \\ 

$3_{1} \# 8_{5}^m$ & Invertible \\ 

$3_{1}^m \# 8_{5}$ & Invertible \\ 

$3_{1}^m \# 8_{5}^m$ & Invertible \\ \addlinespace[0.455092817969205em]

$3_{1} \# 8_{6}$ & Invertible \\ 

$3_{1} \# 8_{6}^m$ & Invertible \\ 

$3_{1}^m \# 8_{6}$ & Invertible \\ 

$3_{1}^m \# 8_{6}^m$ & Invertible \\ 
\bottomrule
\end{tabular}
\hfill\hfill\end{center}
\end{table}

\newpage
\begin{table}[h]
 \caption{\label{CompositeKnotTable 1} Composite Knot Types, Part 2 of 6}

\begin{center}
\hfill\begin{tabular}[t]{l@{\hspace{0.1in}}r} \toprule
Knot & Symmetry\\ \midrule
$3_{1} \# 8_{7}$ & Invertible \\ 

$3_{1} \# 8_{7}^m$ & Invertible \\ 

$3_{1}^m \# 8_{7}$ & Invertible \\ 

$3_{1}^m \# 8_{7}^m$ & Invertible \\ \addlinespace[0.540748556249757em]

$3_{1} \# 8_{8}$ & Invertible \\ 

$3_{1} \# 8_{8}^m$ & Invertible \\ 

$3_{1}^m \# 8_{8}$ & Invertible \\ 

$3_{1}^m \# 8_{8}^m$ & Invertible \\ \addlinespace[0.540748556249757em]

$3_{1} \# 8_{9}$ & Invertible \\ 

$3_{1}^m \# 8_{9}$ & Invertible \\ \addlinespace[0.540748556249757em]

$3_{1} \# 8_{10}$ & Invertible \\ 

$3_{1} \# 8_{10}^m$ & Invertible \\ 

$3_{1}^m \# 8_{10}$ & Invertible \\ 

$3_{1}^m \# 8_{10}^m$ & Invertible \\ \addlinespace[0.540748556249757em]

$3_{1} \# 8_{11}$ & Invertible \\ 

$3_{1} \# 8_{11}^m$ & Invertible \\ 

$3_{1}^m \# 8_{11}$ & Invertible \\ 

$3_{1}^m \# 8_{11}^m$ & Invertible \\ \addlinespace[0.540748556249757em]

$3_{1} \# 8_{12}$ & Invertible \\ 

$3_{1}^m \# 8_{12}$ & Invertible \\ \addlinespace[0.540748556249757em]

$3_{1} \# 8_{13}$ & Invertible \\ 

$3_{1} \# 8_{13}^m$ & Invertible \\ 

$3_{1}^m \# 8_{13}$ & Invertible \\ 

$3_{1}^m \# 8_{13}^m$ & Invertible \\ \addlinespace[0.540748556249757em]

$3_{1} \# 8_{14}$ & Invertible \\ 

$3_{1} \# 8_{14}^m$ & Invertible \\ 

$3_{1}^m \# 8_{14}$ & Invertible \\ 

$3_{1}^m \# 8_{14}^m$ & Invertible \\ \addlinespace[0.540748556249757em]

$3_{1} \# 8_{15}$ & Invertible \\ 

$3_{1} \# 8_{15}^m$ & Invertible \\ 

$3_{1}^m \# 8_{15}$ & Invertible \\ 

$3_{1}^m \# 8_{15}^m$ & Invertible \\ \addlinespace[0.540748556249757em]

$3_{1} \# 8_{16}$ & Invertible \\ 

$3_{1} \# 8_{16}^m$ & Invertible \\ 
\bottomrule
\end{tabular}
\hfill\begin{tabular}[t]{l@{\hspace{0.1in}}r} \toprule
Knot & Symmetry\\ \midrule
$3_{1}^m \# 8_{16}$ & Invertible \\ 

$3_{1}^m \# 8_{16}^m$ & Invertible \\ \addlinespace[0.50840017917916em]

$3_{1} \# 8_{17}$ & None \\ 

$3_{1} \# 8_{17}^r$ & None \\ 

$3_{1}^m \# 8_{17}$ & None \\ 

$3_{1}^m \# 8_{17}^r$ & None \\ \addlinespace[0.50840017917916em]

$3_{1} \# 8_{18}$ & Invertible \\ 

$3_{1}^m \# 8_{18}$ & Invertible \\ \addlinespace[0.50840017917916em]

$3_{1} \# 8_{19}$ & Invertible \\ 

$3_{1} \# 8_{19}^m$ & Invertible \\ 

$3_{1}^m \# 8_{19}$ & Invertible \\ 

$3_{1}^m \# 8_{19}^m$ & Invertible \\ \addlinespace[0.50840017917916em]

$3_{1} \# 8_{20}$ & Invertible \\ 

$3_{1} \# 8_{20}^m$ & Invertible \\ 

$3_{1}^m \# 8_{20}$ & Invertible \\ 

$3_{1}^m \# 8_{20}^m$ & Invertible \\ \addlinespace[0.50840017917916em]

$3_{1} \# 8_{21}$ & Invertible \\ 

$3_{1} \# 8_{21}^m$ & Invertible \\ 

$3_{1}^m \# 8_{21}$ & Invertible \\ 

$3_{1}^m \# 8_{21}^m$ & Invertible \\ \addlinespace[0.50840017917916em]

$4_{1} \# 7_{1}$ & Invertible \\ 

$4_{1} \# 7_{1}^m$ & Invertible \\ \addlinespace[0.50840017917916em]

$4_{1} \# 7_{2}$ & Invertible \\ 

$4_{1} \# 7_{2}^m$ & Invertible \\ \addlinespace[0.50840017917916em]

$4_{1} \# 7_{3}$ & Invertible \\ 

$4_{1} \# 7_{3}^m$ & Invertible \\ \addlinespace[0.50840017917916em]

$4_{1} \# 7_{4}$ & Invertible \\ 

$4_{1} \# 7_{4}^m$ & Invertible \\ \addlinespace[0.50840017917916em]

$4_{1} \# 7_{5}$ & Invertible \\ 

$4_{1} \# 7_{5}^m$ & Invertible \\ \addlinespace[0.50840017917916em]

$4_{1} \# 7_{6}$ & Invertible \\ 

$4_{1} \# 7_{6}^m$ & Invertible \\ \addlinespace[0.50840017917916em]

$4_{1} \# 7_{7}$ & Invertible \\ 
\bottomrule
\end{tabular}
\hfill\begin{tabular}[t]{l@{\hspace{0.1in}}r} \toprule
Knot & Symmetry\\ \midrule
$4_{1} \# 7_{7}^m$ & Invertible \\ \addlinespace[0.540748556249757em]

$5_{1} \# 6_{1}$ & Invertible \\ 

$5_{1} \# 6_{1}^m$ & Invertible \\ 

$5_{1}^m \# 6_{1}$ & Invertible \\ 

$5_{1}^m \# 6_{1}^m$ & Invertible \\ \addlinespace[0.540748556249757em]

$5_{1} \# 6_{2}$ & Invertible \\ 

$5_{1} \# 6_{2}^m$ & Invertible \\ 

$5_{1}^m \# 6_{2}$ & Invertible \\ 

$5_{1}^m \# 6_{2}^m$ & Invertible \\ \addlinespace[0.540748556249757em]

$5_{1} \# 6_{3}$ & Invertible \\ 

$5_{1}^m \# 6_{3}$ & Invertible \\ \addlinespace[0.540748556249757em]

$5_{2} \# 6_{1}$ & Invertible \\ 

$5_{2} \# 6_{1}^m$ & Invertible \\ 

$5_{2}^m \# 6_{1}$ & Invertible \\ 

$5_{2}^m \# 6_{1}^m$ & Invertible \\ \addlinespace[0.540748556249757em]

$5_{2} \# 6_{2}$ & Invertible \\ 

$5_{2} \# 6_{2}^m$ & Invertible \\ 

$5_{2}^m \# 6_{2}$ & Invertible \\ 

$5_{2}^m \# 6_{2}^m$ & Invertible \\ \addlinespace[0.540748556249757em]

$5_{2} \# 6_{3}$ & Invertible \\ 

$5_{2}^m \# 6_{3}$ & Invertible \\ \addlinespace[0.540748556249757em]

$3_{1} \# 3_{1} \# 5_{1}$ & Invertible \\ 

$3_{1} \# 3_{1} \# 5_{1}^m$ & Invertible \\ 

$3_{1} \# 3_{1}^m \# 5_{1}$ & Invertible \\ 

$3_{1} \# 3_{1}^m \# 5_{1}^m$ & Invertible \\ 

$3_{1}^m \# 3_{1}^m \# 5_{1}$ & Invertible \\ 

$3_{1}^m \# 3_{1}^m \# 5_{1}^m$ & Invertible \\ \addlinespace[0.540748556249757em]

$3_{1} \# 3_{1} \# 5_{2}$ & Invertible \\ 

$3_{1} \# 3_{1} \# 5_{2}^m$ & Invertible \\ 

$3_{1} \# 3_{1}^m \# 5_{2}$ & Invertible \\ 

$3_{1} \# 3_{1}^m \# 5_{2}^m$ & Invertible \\ 

$3_{1}^m \# 3_{1}^m \# 5_{2}$ & Invertible \\ 

$3_{1}^m \# 3_{1}^m \# 5_{2}^m$ & Invertible \\ \addlinespace[0.540748556249757em]

$3_{1} \# 4_{1} \# 4_{1}$ & Invertible \\ 
\bottomrule
\end{tabular}
\hfill\hfill\end{center}
\end{table}

\newpage
\begin{table}[h]
 \caption{\label{CompositeKnotTable 2} Composite Knot Types, Part 3 of 6}

\begin{center}
\hfill\begin{tabular}[t]{l@{\hspace{0.1in}}r} \toprule
Knot & Symmetry\\ \midrule
$3_{1}^m \# 4_{1} \# 4_{1}$ & Invertible \\ \addlinespace[0.455092817969205em]\cmidrule(r{0.5em}l{0.5em}){1-2}\addlinespace[0.455092817969205em]

$3_{1} \# 9_{1}$ & Invertible \\ 

$3_{1} \# 9_{1}^m$ & Invertible \\ 

$3_{1}^m \# 9_{1}$ & Invertible \\ 

$3_{1}^m \# 9_{1}^m$ & Invertible \\ \addlinespace[0.455092817969205em]

$3_{1} \# 9_{2}$ & Invertible \\ 

$3_{1} \# 9_{2}^m$ & Invertible \\ 

$3_{1}^m \# 9_{2}$ & Invertible \\ 

$3_{1}^m \# 9_{2}^m$ & Invertible \\ \addlinespace[0.455092817969205em]

$3_{1} \# 9_{3}$ & Invertible \\ 

$3_{1} \# 9_{3}^m$ & Invertible \\ 

$3_{1}^m \# 9_{3}$ & Invertible \\ 

$3_{1}^m \# 9_{3}^m$ & Invertible \\ \addlinespace[0.455092817969205em]

$3_{1} \# 9_{4}$ & Invertible \\ 

$3_{1} \# 9_{4}^m$ & Invertible \\ 

$3_{1}^m \# 9_{4}$ & Invertible \\ 

$3_{1}^m \# 9_{4}^m$ & Invertible \\ \addlinespace[0.455092817969205em]

$3_{1} \# 9_{5}$ & Invertible \\ 

$3_{1} \# 9_{5}^m$ & Invertible \\ 

$3_{1}^m \# 9_{5}$ & Invertible \\ 

$3_{1}^m \# 9_{5}^m$ & Invertible \\ \addlinespace[0.455092817969205em]

$3_{1} \# 9_{6}$ & Invertible \\ 

$3_{1} \# 9_{6}^m$ & Invertible \\ 

$3_{1}^m \# 9_{6}$ & Invertible \\ 

$3_{1}^m \# 9_{6}^m$ & Invertible \\ \addlinespace[0.455092817969205em]

$3_{1} \# 9_{7}$ & Invertible \\ 

$3_{1} \# 9_{7}^m$ & Invertible \\ 

$3_{1}^m \# 9_{7}$ & Invertible \\ 

$3_{1}^m \# 9_{7}^m$ & Invertible \\ \addlinespace[0.455092817969205em]

$3_{1} \# 9_{8}$ & Invertible \\ 

$3_{1} \# 9_{8}^m$ & Invertible \\ 

$3_{1}^m \# 9_{8}$ & Invertible \\ 

$3_{1}^m \# 9_{8}^m$ & Invertible \\ \addlinespace[0.455092817969205em]

$3_{1} \# 9_{9}$ & Invertible \\ 
\bottomrule
\end{tabular}
\hfill\begin{tabular}[t]{l@{\hspace{0.1in}}r} \toprule
Knot & Symmetry\\ \midrule
$3_{1} \# 9_{9}^m$ & Invertible \\ 

$3_{1}^m \# 9_{9}$ & Invertible \\ 

$3_{1}^m \# 9_{9}^m$ & Invertible \\ \addlinespace[0.454084935335509em]

$3_{1} \# 9_{10}$ & Invertible \\ 

$3_{1} \# 9_{10}^m$ & Invertible \\ 

$3_{1}^m \# 9_{10}$ & Invertible \\ 

$3_{1}^m \# 9_{10}^m$ & Invertible \\ \addlinespace[0.454084935335509em]

$3_{1} \# 9_{11}$ & Invertible \\ 

$3_{1} \# 9_{11}^m$ & Invertible \\ 

$3_{1}^m \# 9_{11}$ & Invertible \\ 

$3_{1}^m \# 9_{11}^m$ & Invertible \\ \addlinespace[0.454084935335509em]

$3_{1} \# 9_{12}$ & Invertible \\ 

$3_{1} \# 9_{12}^m$ & Invertible \\ 

$3_{1}^m \# 9_{12}$ & Invertible \\ 

$3_{1}^m \# 9_{12}^m$ & Invertible \\ \addlinespace[0.454084935335509em]

$3_{1} \# 9_{13}$ & Invertible \\ 

$3_{1} \# 9_{13}^m$ & Invertible \\ 

$3_{1}^m \# 9_{13}$ & Invertible \\ 

$3_{1}^m \# 9_{13}^m$ & Invertible \\ \addlinespace[0.454084935335509em]

$3_{1} \# 9_{14}$ & Invertible \\ 

$3_{1} \# 9_{14}^m$ & Invertible \\ 

$3_{1}^m \# 9_{14}$ & Invertible \\ 

$3_{1}^m \# 9_{14}^m$ & Invertible \\ \addlinespace[0.454084935335509em]

$3_{1} \# 9_{15}$ & Invertible \\ 

$3_{1} \# 9_{15}^m$ & Invertible \\ 

$3_{1}^m \# 9_{15}$ & Invertible \\ 

$3_{1}^m \# 9_{15}^m$ & Invertible \\ \addlinespace[0.454084935335509em]

$3_{1} \# 9_{16}$ & Invertible \\ 

$3_{1} \# 9_{16}^m$ & Invertible \\ 

$3_{1}^m \# 9_{16}$ & Invertible \\ 

$3_{1}^m \# 9_{16}^m$ & Invertible \\ \addlinespace[0.454084935335509em]

$3_{1} \# 9_{17}$ & Invertible \\ 

$3_{1} \# 9_{17}^m$ & Invertible \\ 

$3_{1}^m \# 9_{17}$ & Invertible \\ 

$3_{1}^m \# 9_{17}^m$ & Invertible \\ 
\bottomrule
\end{tabular}
\hfill\begin{tabular}[t]{l@{\hspace{0.1in}}r} \toprule
Knot & Symmetry\\ \midrule
$3_{1} \# 9_{18}$ & Invertible \\ 

$3_{1} \# 9_{18}^m$ & Invertible \\ 

$3_{1}^m \# 9_{18}$ & Invertible \\ 

$3_{1}^m \# 9_{18}^m$ & Invertible \\ \addlinespace[0.454084935335509em]

$3_{1} \# 9_{19}$ & Invertible \\ 

$3_{1} \# 9_{19}^m$ & Invertible \\ 

$3_{1}^m \# 9_{19}$ & Invertible \\ 

$3_{1}^m \# 9_{19}^m$ & Invertible \\ \addlinespace[0.454084935335509em]

$3_{1} \# 9_{20}$ & Invertible \\ 

$3_{1} \# 9_{20}^m$ & Invertible \\ 

$3_{1}^m \# 9_{20}$ & Invertible \\ 

$3_{1}^m \# 9_{20}^m$ & Invertible \\ \addlinespace[0.454084935335509em]

$3_{1} \# 9_{21}$ & Invertible \\ 

$3_{1} \# 9_{21}^m$ & Invertible \\ 

$3_{1}^m \# 9_{21}$ & Invertible \\ 

$3_{1}^m \# 9_{21}^m$ & Invertible \\ \addlinespace[0.454084935335509em]

$3_{1} \# 9_{22}$ & Invertible \\ 

$3_{1} \# 9_{22}^m$ & Invertible \\ 

$3_{1}^m \# 9_{22}$ & Invertible \\ 

$3_{1}^m \# 9_{22}^m$ & Invertible \\ \addlinespace[0.454084935335509em]

$3_{1} \# 9_{23}$ & Invertible \\ 

$3_{1} \# 9_{23}^m$ & Invertible \\ 

$3_{1}^m \# 9_{23}$ & Invertible \\ 

$3_{1}^m \# 9_{23}^m$ & Invertible \\ \addlinespace[0.454084935335509em]

$3_{1} \# 9_{24}$ & Invertible \\ 

$3_{1} \# 9_{24}^m$ & Invertible \\ 

$3_{1}^m \# 9_{24}$ & Invertible \\ 

$3_{1}^m \# 9_{24}^m$ & Invertible \\ \addlinespace[0.454084935335509em]

$3_{1} \# 9_{25}$ & Invertible \\ 

$3_{1} \# 9_{25}^m$ & Invertible \\ 

$3_{1}^m \# 9_{25}$ & Invertible \\ 

$3_{1}^m \# 9_{25}^m$ & Invertible \\ \addlinespace[0.454084935335509em]

$3_{1} \# 9_{26}$ & Invertible \\ 

$3_{1} \# 9_{26}^m$ & Invertible \\ 

$3_{1}^m \# 9_{26}$ & Invertible \\ 
\bottomrule
\end{tabular}
\hfill\hfill\end{center}
\end{table}

\newpage
\begin{table}[h]
 \caption{\label{CompositeKnotTable 3} Composite Knot Types, Part 4 of 6}

\begin{center}
\hfill\begin{tabular}[t]{l@{\hspace{0.1in}}r} \toprule
Knot & Symmetry\\ \midrule
$3_{1}^m \# 9_{26}^m$ & Invertible \\ \addlinespace[0.518953409372206em]

$3_{1} \# 9_{27}$ & Invertible \\ 

$3_{1} \# 9_{27}^m$ & Invertible \\ 

$3_{1}^m \# 9_{27}$ & Invertible \\ 

$3_{1}^m \# 9_{27}^m$ & Invertible \\ \addlinespace[0.518953409372206em]

$3_{1} \# 9_{28}$ & Invertible \\ 

$3_{1} \# 9_{28}^m$ & Invertible \\ 

$3_{1}^m \# 9_{28}$ & Invertible \\ 

$3_{1}^m \# 9_{28}^m$ & Invertible \\ \addlinespace[0.518953409372206em]

$3_{1} \# 9_{29}$ & Invertible \\ 

$3_{1} \# 9_{29}^m$ & Invertible \\ 

$3_{1}^m \# 9_{29}$ & Invertible \\ 

$3_{1}^m \# 9_{29}^m$ & Invertible \\ \addlinespace[0.518953409372206em]

$3_{1} \# 9_{30}$ & Invertible \\ 

$3_{1} \# 9_{30}^m$ & Invertible \\ 

$3_{1}^m \# 9_{30}$ & Invertible \\ 

$3_{1}^m \# 9_{30}^m$ & Invertible \\ \addlinespace[0.518953409372206em]

$3_{1} \# 9_{31}$ & Invertible \\ 

$3_{1} \# 9_{31}^m$ & Invertible \\ 

$3_{1}^m \# 9_{31}$ & Invertible \\ 

$3_{1}^m \# 9_{31}^m$ & Invertible \\ \addlinespace[0.518953409372206em]

$3_{1} \# 9_{32}$ & None \\ 

$3_{1} \# 9_{32}^r$ & None \\ 

$3_{1} \# 9_{32}^m$ & None \\ 

$3_{1} \# 9_{32}^{rm}$ & None \\ 

$3_{1}^m \# 9_{32}$ & None \\ 

$3_{1}^m \# 9_{32}^r$ & None \\ 

$3_{1}^m \# 9_{32}^m$ & None \\ 

$3_{1}^m \# 9_{32}^{rm}$ & None \\ \addlinespace[0.518953409372206em]

$3_{1} \# 9_{33}$ & None \\ 

$3_{1} \# 9_{33}^r$ & None \\ 

$3_{1} \# 9_{33}^m$ & None \\ 

$3_{1} \# 9_{33}^{rm}$ & None \\ 

$3_{1}^m \# 9_{33}$ & None \\ 

$3_{1}^m \# 9_{33}^r$ & None \\

\bottomrule
\end{tabular}
\hfill\begin{tabular}[t]{l@{\hspace{0.1in}}r} \toprule
Knot & Symmetry\\ \midrule
$3_{1}^m \# 9_{33}^m$ & None \\ 

$3_{1}^m \# 9_{33}^{rm}$ & None \\ \addlinespace[0.608341438027759em]

$3_{1} \# 9_{34}$ & Invertible \\ 

$3_{1} \# 9_{34}^m$ & Invertible \\ 

$3_{1}^m \# 9_{34}$ & Invertible \\ 

$3_{1}^m \# 9_{34}^m$ & Invertible \\ \addlinespace[0.608341438027759em]

$3_{1} \# 9_{35}$ & Invertible \\ 

$3_{1} \# 9_{35}^m$ & Invertible \\ 

$3_{1}^m \# 9_{35}$ & Invertible \\ 

$3_{1}^m \# 9_{35}^m$ & Invertible \\ \addlinespace[0.608341438027759em]

$3_{1} \# 9_{36}$ & Invertible \\ 

$3_{1} \# 9_{36}^m$ & Invertible \\ 

$3_{1}^m \# 9_{36}$ & Invertible \\ 

$3_{1}^m \# 9_{36}^m$ & Invertible \\ \addlinespace[0.608341438027759em]

$3_{1} \# 9_{37}$ & Invertible \\ 

$3_{1} \# 9_{37}^m$ & Invertible \\ 

$3_{1}^m \# 9_{37}$ & Invertible \\ 

$3_{1}^m \# 9_{37}^m$ & Invertible \\ \addlinespace[0.608341438027759em]

$3_{1} \# 9_{38}$ & Invertible \\ 

$3_{1} \# 9_{38}^m$ & Invertible \\ 

$3_{1}^m \# 9_{38}$ & Invertible \\ 

$3_{1}^m \# 9_{38}^m$ & Invertible \\ \addlinespace[0.608341438027759em]

$3_{1} \# 9_{39}$ & Invertible \\ 

$3_{1} \# 9_{39}^m$ & Invertible \\ 

$3_{1}^m \# 9_{39}$ & Invertible \\ 

$3_{1}^m \# 9_{39}^m$ & Invertible \\ \addlinespace[0.608341438027759em]

$3_{1} \# 9_{40}$ & Invertible \\ 

$3_{1} \# 9_{40}^m$ & Invertible \\ 

$3_{1}^m \# 9_{40}$ & Invertible \\ 

$3_{1}^m \# 9_{40}^m$ & Invertible \\ \addlinespace[0.608341438027759em]

$3_{1} \# 9_{41}$ & Invertible \\ 

$3_{1} \# 9_{41}^m$ & Invertible \\ 

$3_{1}^m \# 9_{41}$ & Invertible \\ 

$3_{1}^m \# 9_{41}^m$ & Invertible \\ 
\bottomrule
\end{tabular}
\hfill\begin{tabular}[t]{l@{\hspace{0.1in}}r} \toprule
Knot & Symmetry\\ \midrule
$3_{1} \# 9_{42}$ & Invertible \\ 

$3_{1} \# 9_{42}^m$ & Invertible \\ 

$3_{1}^m \# 9_{42}$ & Invertible \\ 

$3_{1}^m \# 9_{42}^m$ & Invertible \\ \addlinespace[0.608341438027759em]

$3_{1} \# 9_{43}$ & Invertible \\ 

$3_{1} \# 9_{43}^m$ & Invertible \\ 

$3_{1}^m \# 9_{43}$ & Invertible \\ 

$3_{1}^m \# 9_{43}^m$ & Invertible \\ \addlinespace[0.608341438027759em]

$3_{1} \# 9_{44}$ & Invertible \\ 

$3_{1} \# 9_{44}^m$ & Invertible \\ 

$3_{1}^m \# 9_{44}$ & Invertible \\ 

$3_{1}^m \# 9_{44}^m$ & Invertible \\ \addlinespace[0.608341438027759em]

$3_{1} \# 9_{45}$ & Invertible \\ 

$3_{1} \# 9_{45}^m$ & Invertible \\ 

$3_{1}^m \# 9_{45}$ & Invertible \\ 

$3_{1}^m \# 9_{45}^m$ & Invertible \\ \addlinespace[0.608341438027759em]

$3_{1} \# 9_{46}$ & Invertible \\ 

$3_{1} \# 9_{46}^m$ & Invertible \\ 

$3_{1}^m \# 9_{46}$ & Invertible \\ 

$3_{1}^m \# 9_{46}^m$ & Invertible \\ \addlinespace[0.608341438027759em]

$3_{1} \# 9_{47}$ & Invertible \\ 

$3_{1} \# 9_{47}^m$ & Invertible \\ 

$3_{1}^m \# 9_{47}$ & Invertible \\ 

$3_{1}^m \# 9_{47}^m$ & Invertible \\ \addlinespace[0.608341438027759em]

$3_{1} \# 9_{48}$ & Invertible \\ 

$3_{1} \# 9_{48}^m$ & Invertible \\ 

$3_{1}^m \# 9_{48}$ & Invertible \\ 

$3_{1}^m \# 9_{48}^m$ & Invertible \\ \addlinespace[0.608341438027759em]

$3_{1} \# 9_{49}$ & Invertible \\ 

$3_{1} \# 9_{49}^m$ & Invertible \\ 

$3_{1}^m \# 9_{49}$ & Invertible \\ 

$3_{1}^m \# 9_{49}^m$ & Invertible \\ \addlinespace[0.608341438027759em]

$4_{1} \# 8_{1}$ & Invertible \\ 

$4_{1} \# 8_{1}^m$ & Invertible \\ 
\bottomrule
\end{tabular}
\hfill\hfill\end{center}
\end{table}

\newpage
\begin{table}[h]
 \caption{\label{CompositeKnotTable 4} Composite Knot Types, Part 5 of 6}

\begin{center}
\hfill\begin{tabular}[t]{l@{\hspace{0.1in}}r} \toprule
Knot & Symmetry\\ \midrule
$4_{1} \# 8_{2}$ & Invertible \\ 

$4_{1} \# 8_{2}^m$ & Invertible \\ \addlinespace[0.504054630456227em]

$4_{1} \# 8_{3}$ & Full \\ \addlinespace[0.504054630456227em]

$4_{1} \# 8_{4}$ & Invertible \\ 

$4_{1} \# 8_{4}^m$ & Invertible \\ \addlinespace[0.504054630456227em]

$4_{1} \# 8_{5}$ & Invertible \\ 

$4_{1} \# 8_{5}^m$ & Invertible \\ \addlinespace[0.504054630456227em]

$4_{1} \# 8_{6}$ & Invertible \\ 

$4_{1} \# 8_{6}^m$ & Invertible \\ \addlinespace[0.504054630456227em]

$4_{1} \# 8_{7}$ & Invertible \\ 

$4_{1} \# 8_{7}^m$ & Invertible \\ \addlinespace[0.504054630456227em]

$4_{1} \# 8_{8}$ & Invertible \\ 

$4_{1} \# 8_{8}^m$ & Invertible \\ \addlinespace[0.504054630456227em]

$4_{1} \# 8_{9}$ & Full \\ \addlinespace[0.504054630456227em]

$4_{1} \# 8_{10}$ & Invertible \\ 

$4_{1} \# 8_{10}^m$ & Invertible \\ \addlinespace[0.504054630456227em]

$4_{1} \# 8_{11}$ & Invertible \\ 

$4_{1} \# 8_{11}^m$ & Invertible \\ \addlinespace[0.504054630456227em]

$4_{1} \# 8_{12}$ & Full \\ \addlinespace[0.504054630456227em]

$4_{1} \# 8_{13}$ & Invertible \\ 

$4_{1} \# 8_{13}^m$ & Invertible \\ \addlinespace[0.504054630456227em]

$4_{1} \# 8_{14}$ & Invertible \\ 

$4_{1} \# 8_{14}^m$ & Invertible \\ \addlinespace[0.504054630456227em]

$4_{1} \# 8_{15}$ & Invertible \\ 

$4_{1} \# 8_{15}^m$ & Invertible \\ \addlinespace[0.504054630456227em]

$4_{1} \# 8_{16}$ & Invertible \\ 

$4_{1} \# 8_{16}^m$ & Invertible \\ \addlinespace[0.504054630456227em]

$4_{1} \# 8_{17}$ & (-) Amphichiral \\ 

$4_{1} \# 8_{17}^r$ & (-) Amphichiral \\ \addlinespace[0.504054630456227em]

$4_{1} \# 8_{18}$ & Full \\ \addlinespace[0.504054630456227em]

$4_{1} \# 8_{19}$ & Invertible \\ 
\bottomrule
\end{tabular}
\hfill\begin{tabular}[t]{l@{\hspace{0.1in}}r} \toprule
Knot & Symmetry\\ \midrule
$4_{1} \# 8_{19}^m$ & Invertible \\ \addlinespace[0.486673912699349em]

$4_{1} \# 8_{20}$ & Invertible \\ 

$4_{1} \# 8_{20}^m$ & Invertible \\ \addlinespace[0.486673912699349em]

$4_{1} \# 8_{21}$ & Invertible \\ 

$4_{1} \# 8_{21}^m$ & Invertible \\ \addlinespace[0.486673912699349em]

$5_{1} \# 7_{1}$ & Invertible \\ 

$5_{1} \# 7_{1}^m$ & Invertible \\ 

$5_{1}^m \# 7_{1}$ & Invertible \\ 

$5_{1}^m \# 7_{1}^m$ & Invertible \\ \addlinespace[0.486673912699349em]

$5_{1} \# 7_{2}$ & Invertible \\ 

$5_{1} \# 7_{2}^m$ & Invertible \\ 

$5_{1}^m \# 7_{2}$ & Invertible \\ 

$5_{1}^m \# 7_{2}^m$ & Invertible \\ \addlinespace[0.486673912699349em]

$5_{1} \# 7_{3}$ & Invertible \\ 

$5_{1} \# 7_{3}^m$ & Invertible \\ 

$5_{1}^m \# 7_{3}$ & Invertible \\ 

$5_{1}^m \# 7_{3}^m$ & Invertible \\ \addlinespace[0.486673912699349em]

$5_{1} \# 7_{4}$ & Invertible \\ 

$5_{1} \# 7_{4}^m$ & Invertible \\ 

$5_{1}^m \# 7_{4}$ & Invertible \\ 

$5_{1}^m \# 7_{4}^m$ & Invertible \\ \addlinespace[0.486673912699349em]

$5_{1} \# 7_{5}$ & Invertible \\ 

$5_{1} \# 7_{5}^m$ & Invertible \\ 

$5_{1}^m \# 7_{5}$ & Invertible \\ 

$5_{1}^m \# 7_{5}^m$ & Invertible \\ \addlinespace[0.486673912699349em]

$5_{1} \# 7_{6}$ & Invertible \\ 

$5_{1} \# 7_{6}^m$ & Invertible \\ 

$5_{1}^m \# 7_{6}$ & Invertible \\ 

$5_{1}^m \# 7_{6}^m$ & Invertible \\ \addlinespace[0.486673912699349em]

$5_{1} \# 7_{7}$ & Invertible \\ 

$5_{1} \# 7_{7}^m$ & Invertible \\ 

$5_{1}^m \# 7_{7}$ & Invertible \\ 

$5_{1}^m \# 7_{7}^m$ & Invertible \\ \addlinespace[0.486673912699349em]

$5_{2} \# 7_{1}$ & Invertible \\ 
\bottomrule
\end{tabular}
\hfill\begin{tabular}[t]{l@{\hspace{0.1in}}r} \toprule
Knot & Symmetry\\ \midrule
$5_{2} \# 7_{1}^m$ & Invertible \\ 

$5_{2}^m \# 7_{1}$ & Invertible \\ 

$5_{2}^m \# 7_{1}^m$ & Invertible \\ \addlinespace[0.608341438027759em]

$5_{2} \# 7_{2}$ & Invertible \\ 

$5_{2} \# 7_{2}^m$ & Invertible \\ 

$5_{2}^m \# 7_{2}$ & Invertible \\ 

$5_{2}^m \# 7_{2}^m$ & Invertible \\ \addlinespace[0.608341438027759em]

$5_{2} \# 7_{3}$ & Invertible \\ 

$5_{2} \# 7_{3}^m$ & Invertible \\ 

$5_{2}^m \# 7_{3}$ & Invertible \\ 

$5_{2}^m \# 7_{3}^m$ & Invertible \\ \addlinespace[0.608341438027759em]

$5_{2} \# 7_{4}$ & Invertible \\ 

$5_{2} \# 7_{4}^m$ & Invertible \\ 

$5_{2}^m \# 7_{4}$ & Invertible \\ 

$5_{2}^m \# 7_{4}^m$ & Invertible \\ \addlinespace[0.608341438027759em]

$5_{2} \# 7_{5}$ & Invertible \\ 

$5_{2} \# 7_{5}^m$ & Invertible \\ 

$5_{2}^m \# 7_{5}$ & Invertible \\ 

$5_{2}^m \# 7_{5}^m$ & Invertible \\ \addlinespace[0.608341438027759em]

$5_{2} \# 7_{6}$ & Invertible \\ 

$5_{2} \# 7_{6}^m$ & Invertible \\ 

$5_{2}^m \# 7_{6}$ & Invertible \\ 

$5_{2}^m \# 7_{6}^m$ & Invertible \\ \addlinespace[0.608341438027759em]

$5_{2} \# 7_{7}$ & Invertible \\ 

$5_{2} \# 7_{7}^m$ & Invertible \\ 

$5_{2}^m \# 7_{7}$ & Invertible \\ 

$5_{2}^m \# 7_{7}^m$ & Invertible \\ \addlinespace[0.608341438027759em]

$6_{1} \# 6_{1}$ & Invertible \\ 

$6_{1} \# 6_{1}^m$ & Full \\ 

$6_{1}^m \# 6_{1}^m$ & Invertible \\ \addlinespace[0.608341438027759em]

$6_{1} \# 6_{2}$ & Invertible \\ 

$6_{1} \# 6_{2}^m$ & Invertible \\ 

$6_{1}^m \# 6_{2}$ & Invertible \\ 

$6_{1}^m \# 6_{2}^m$ & Invertible \\ 
\bottomrule
\end{tabular}
\hfill\hfill\end{center}
\end{table}

\newpage
\begin{table}[h]
 \caption{\label{CompositeKnotTable 5} Composite Knot Types, Part 6 of 6}

\begin{center}
\hfill\begin{tabular}[t]{l@{\hspace{0.1in}}r} \toprule
Knot & Symmetry\\ \midrule
$6_{1} \# 6_{3}$ & Invertible \\ 

$6_{1}^m \# 6_{3}$ & Invertible \\ \addlinespace[0.486673912699349em]

$6_{2} \# 6_{2}$ & Invertible \\ 

$6_{2} \# 6_{2}^m$ & Full \\ 

$6_{2}^m \# 6_{2}^m$ & Invertible \\ \addlinespace[0.486673912699349em]

$6_{2} \# 6_{3}$ & Invertible \\ 

$6_{2}^m \# 6_{3}$ & Invertible \\ \addlinespace[0.486673912699349em]

$6_{3} \# 6_{3}$ & Full \\ \addlinespace[0.486673912699349em]

$3_{1} \# 3_{1} \# 6_{1}$ & Invertible \\ 

$3_{1} \# 3_{1} \# 6_{1}^m$ & Invertible \\ 

$3_{1} \# 3_{1}^m \# 6_{1}$ & Invertible \\ 

$3_{1} \# 3_{1}^m \# 6_{1}^m$ & Invertible \\ 

$3_{1}^m \# 3_{1}^m \# 6_{1}$ & Invertible \\ 

$3_{1}^m \# 3_{1}^m \# 6_{1}^m$ & Invertible \\ \addlinespace[0.486673912699349em]

$3_{1} \# 3_{1} \# 6_{2}$ & Invertible \\ 

$3_{1} \# 3_{1} \# 6_{2}^m$ & Invertible \\ 

$3_{1} \# 3_{1}^m \# 6_{2}$ & Invertible \\ 

$3_{1} \# 3_{1}^m \# 6_{2}^m$ & Invertible \\ 

$3_{1}^m \# 3_{1}^m \# 6_{2}$ & Invertible \\ 

$3_{1}^m \# 3_{1}^m \# 6_{2}^m$ & Invertible \\ \addlinespace[0.486673912699349em]

$3_{1} \# 3_{1} \# 6_{3}$ & Invertible \\ 

$3_{1} \# 3_{1}^m \# 6_{3}$ & Full \\ 

$3_{1}^m \# 3_{1}^m \# 6_{3}$ & Invertible \\ \addlinespace[0.486673912699349em]

$3_{1} \# 4_{1} \# 5_{1}$ & Invertible \\ 

$3_{1} \# 4_{1} \# 5_{1}^m$ & Invertible \\ 

$3_{1}^m \# 4_{1} \# 5_{1}$ & Invertible \\ 

$3_{1}^m \# 4_{1} \# 5_{1}^m$ & Invertible \\ \addlinespace[0.486673912699349em]

$3_{1} \# 4_{1} \# 5_{2}$ & Invertible \\ 

$3_{1} \# 4_{1} \# 5_{2}^m$ & Invertible \\ 

$3_{1}^m \# 4_{1} \# 5_{2}$ & Invertible \\ 

$3_{1}^m \# 4_{1} \# 5_{2}^m$ & Invertible \\ \addlinespace[0.486673912699349em]

$4_{1} \# 4_{1} \# 4_{1}$ & Full \\ \addlinespace[0.486673912699349em]

$3_{1} \# 3_{1} \# 3_{1} \# 3_{1}$ & Invertible \\ 

$3_{1} \# 3_{1} \# 3_{1} \# 3_{1}^m$ & Invertible \\ 
\bottomrule
\end{tabular}
\hfill\begin{tabular}[t]{l@{\hspace{0.1in}}r} \toprule
Knot & Symmetry\\ \midrule
$3_{1} \# 3_{1} \# 3_{1}^m \# 3_{1}^m$ & Full \\ \addlinespace[0.1em]
$3_{1} \# 3_{1}^m \# 3_{1}^m \# 3_{1}^m$ & Invertible \\ \addlinespace[0.1em]
$3_{1}^m \# 3_{1}^m \# 3_{1}^m \# 3_{1}^m$ & Invertible
 \\ 
\bottomrule
\end{tabular}
\hfill\hfill\end{center}
\end{table}

\end{document}